\documentclass[12pt]{article}

\usepackage{amssymb,amsmath,accents}
\usepackage{amsfonts,amsthm,mathrsfs}
\usepackage{cases} 
\usepackage[dvipdfmx]{graphicx}
\usepackage[dvipdfmx]{color}
\usepackage{subcaption}
\usepackage{url}
\usepackage{xcolor}

\numberwithin{equation}{section}
\newtheorem{thm}{Theorem}[section]
\newtheorem{cor}[thm]{Corollary}
\newtheorem{lemma}[thm]{Lemma}
\newtheorem{prop}[thm]{Proposition}

\newtheorem{definition}[thm]{Definition}
\newtheorem{remark}[thm]{Remark}
\newtheorem{exam}[thm]{Example}

\newcommand{\R}{\mathbb{R}}

\begin{document}
\title{A level-set method for a mean curvature flow \\ with a prescribed boundary}
\vspace{7mm}

\author{
Xingzhi Bian
\thanks{Materials Genome Institute, Shanghai University, 99 Shangda Road, Shanghai, 200444, P.~R.~China. 
E-mail: jdvits@shu.edu.cn. 
The work of the first author was partly supported by China Scholarship Council (CSC), No.~202006890040.}, 
Yoshikazu Giga
\thanks{Graduate School of Mathematical Sciences, The University of Tokyo, Komaba 3-8-1, Meguro, Tokyo 153-8914, Japan. 
E-mail: labgiga@ms.u-tokyo.ac.jp. 
 The work of the second author was partly supported by the Japan Society for the Promotion of Science (JSPS) through the grants Kakenhi: No.~20K20342, No.~19H00639, and by Arithmer Inc., Daikin Industries, Ltd. and Ebara Corporation \ through collaborative grants.} 
and 
Hiroyoshi Mitake
\thanks{Graduate School of Mathematical Sciences, The University of Tokyo, Komaba 3-8-1, Meguro, Tokyo 153-8914, Japan. 
E-mail: mitake@g.ecc.u-tokyo.ac.jp.
The work of the last author was partly supported by JSPS through the grants Kakenhi: No.~22K03382, No.~21H04431, No.~20H01816, No.~19K03580, No.~19H00639.
}}
\date{}

\maketitle

\abstract{
We propose a level-set method for a mean curvature flow whose boundary is prescribed by interpreting the boundary as an obstacle.
 Since the corresponding obstacle problem is globally solvable, our method gives a global-in-time level-set mean curvature flow under a prescribed boundary with no restriction of the profile of an initial hypersurface.
 We show that our solution agrees with a classical mean curvature flow under the Dirichlet condition.
 We moreover prove that our solution agrees with a level-set flow under the Dirichlet condition constructed by P.~Sternberg and W.~P.~Ziemer (1994), where the initial hypersurface is contained in a strictly mean-convex domain and the prescribed boundary is on the boundary of the domain.
} \\



\section{Introduction} \label{S1} 

A level-set method is a powerful tool to track a geometric evolution of a hypersurface like mean curvature flow after it develops singularities.
 Its analytic foundation based on the theory of viscosity solutions was established by \cite{CGG} for a general geometric evolution and independently by \cite{ES} for a mean curvature flow in 1991; see also \cite{G} for later development.
 The goal of this paper is to extend this method for a mean curvature flow with a prescribed boundary.

Let $\Gamma_0$ be a smooth hypersurface embedded in $\mathbb{R}^n$ ($n\geq2$) whose geometric boundary $b\Gamma_0=\Sigma$ is a smooth codimension two compact manifold.
 We consider the mean curvature flow equation for an evolving hypersurface $\{\Gamma_t\}_{t\geq0}$ of the form
\begin{equation} \label{EMCD}
    \left\{
    \begin{alignedat}{2} 
	V &= H \quad\text{on}\quad \Gamma_t, \quad t>0, \\
	b\Gamma_t &= \Sigma, \quad t>0, \\
	\left. \Gamma_t \right|_{t=0} &= \Gamma_0,
    \end{alignedat}
    \right.
\end{equation}
where $V$ denotes the normal velocity of $\Gamma_t$ and $H$ denotes the sum of principal curvatures of $\Gamma_t$ in the direction of a unit normal vector field $\nu$ of $\Gamma_t$;
 $H$ is $n-1$ times mean curvature;
 see Figure \ref{Fr}.
 As in the case when $\Sigma$ is empty (see \cite{Gr89}), the solution $\{\Gamma_t\}$ may develop singularities for $n\geq3$; 
 Figure \ref{Fsing}. 
 See the paragraph right after Corollary \ref{Cav} for a proof. 
Note that even in the case $n=2$, $\{\Gamma_t\}$ may hit the boundary, which is crucially different from the behavior of 
mean curvature flow without obstacles since self-intersection and pinching never happen in two-dimensional setting;
 see \cite{Gr87}. 
 See Figure \ref{Fint}.
\begin{figure}[htb]
\centering
	\begin{minipage}[b]{0.45\linewidth}
\centering
\includegraphics[width=4.5cm]{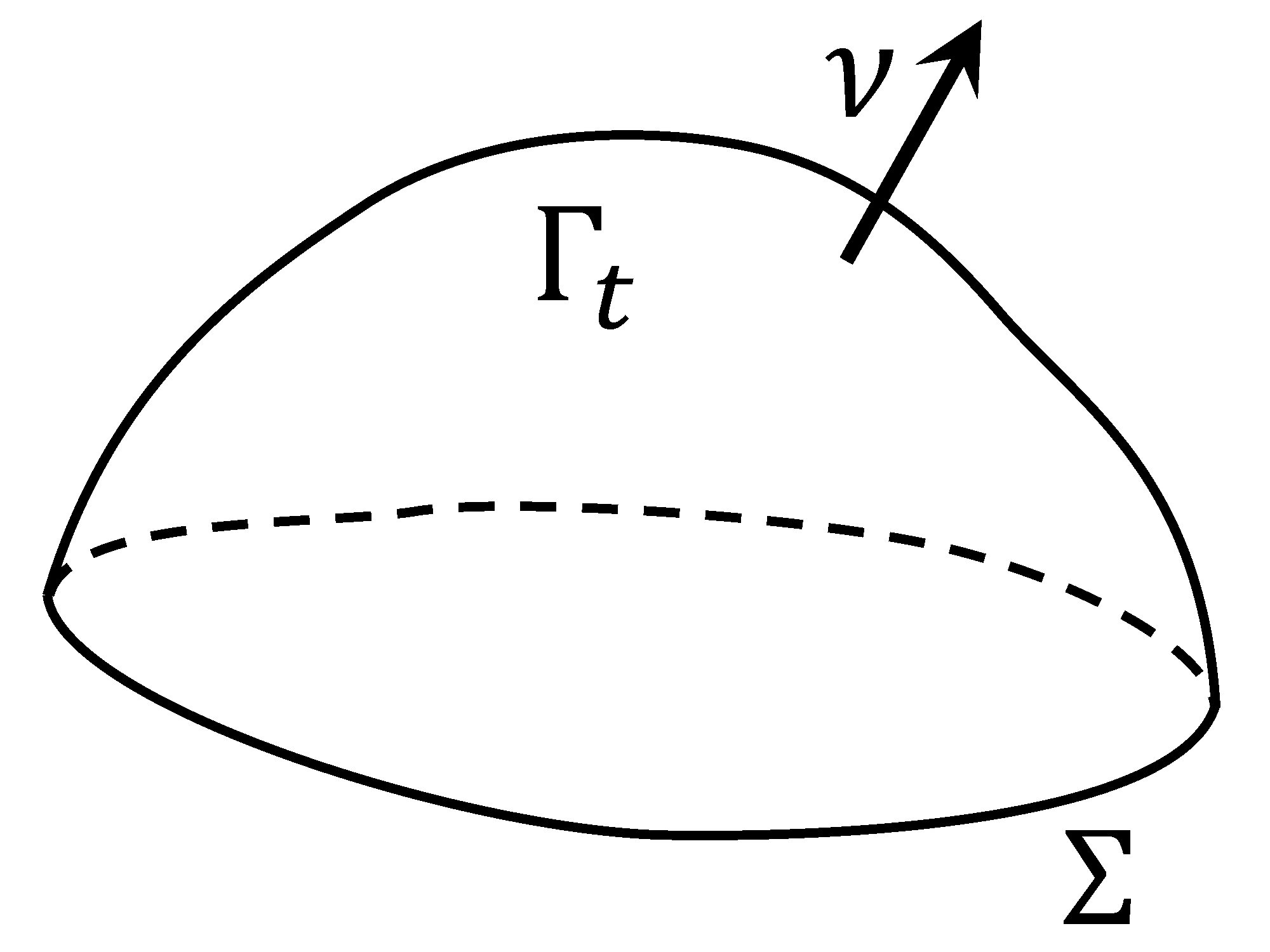}
\caption{prescribed boundary $\Sigma$} \label{Fr}
	\end{minipage}
	\begin{minipage}[b]{0.45\linewidth}
\centering
\includegraphics[width=5.5cm]{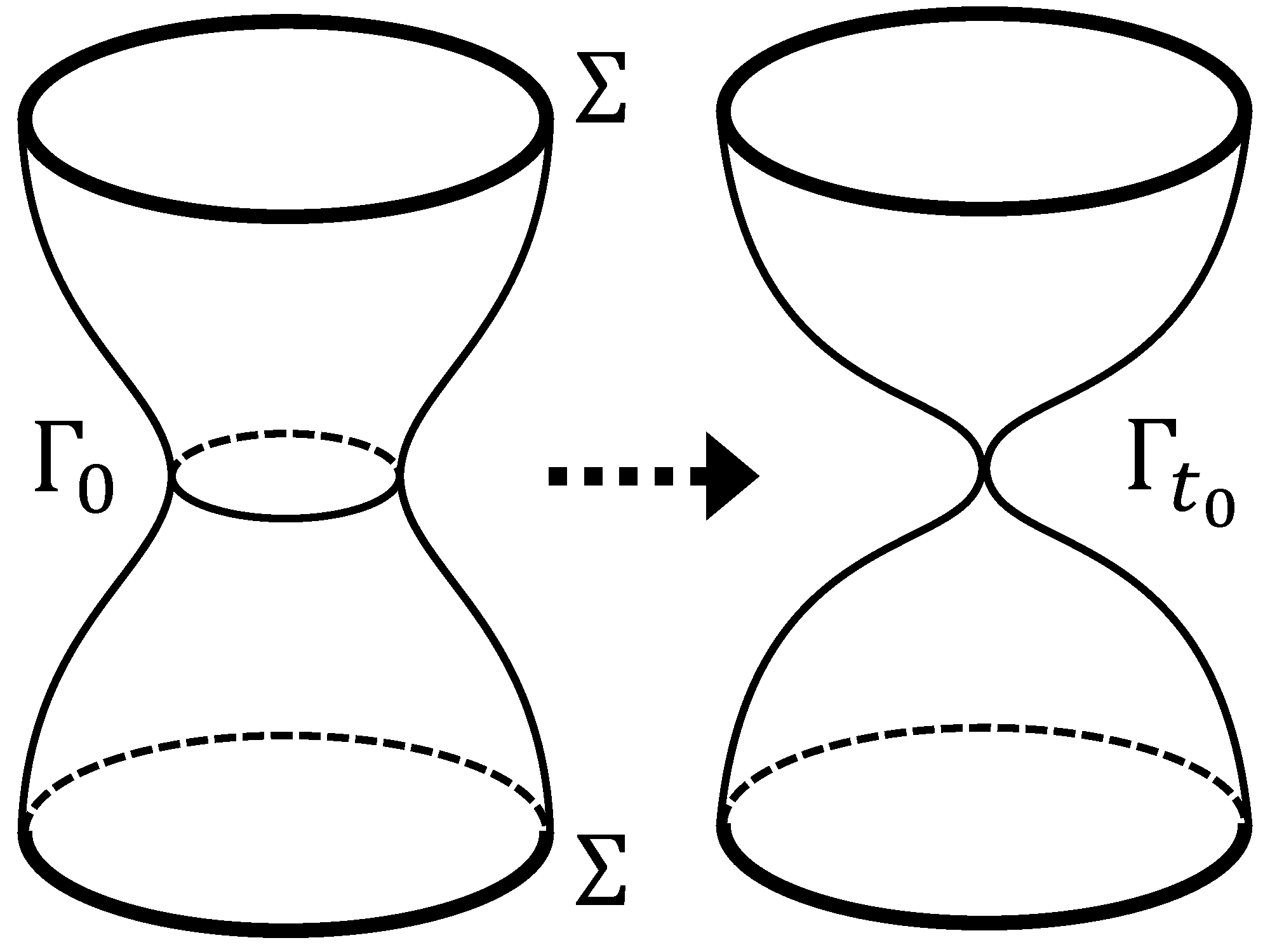}
\caption{pinching at time $t_0$} \label{Fsing}
	\end{minipage}
\end{figure}
\begin{figure}[htb]
\centering
\includegraphics[width=7.5cm]{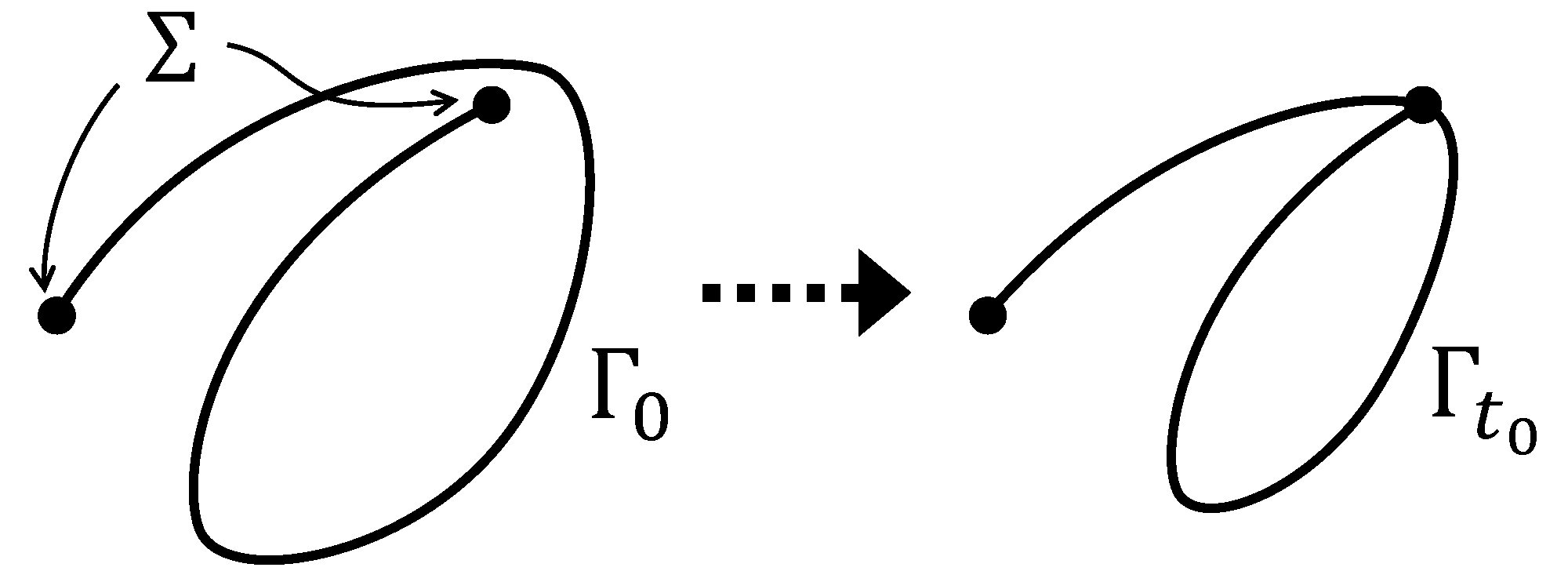}
\caption{hitting $\Sigma$ at time $t_0$} \label{Fint}
\end{figure}

Let us recall the level-set equation for the mean curvature flow equation $V=H$ when $\Sigma$ is empty.
 Let $u$ be a smooth function which is negative inside $\Gamma_t$ and positive outside $\Gamma_t$ with non-vanishing gradient on $\Gamma_t$, i.e., $\nabla u\neq0$ on $\Gamma_t$.
 If $\nu$ is taken $\nu=-\nabla u/|\nabla u|$, then
\[
V = \frac{u_t}{|\nabla u|}, \quad
H = -\operatorname{div}\nu = \operatorname{div}\left(\frac{\nabla u}{|\nabla u|}\right) 
\]
so that $V=H$ is
\begin{equation} \label{Elev}
u_t - |\nabla u| \operatorname{div}\left(\frac{\nabla u}{|\nabla u|}\right) = 0.
\end{equation}
We consider \eqref{Elev} not only on $\Gamma_t$ but also in whole $\mathbb{R}^n$.
 This means that we ask each level-set of $u$ moves by its $(n-1)$ times mean curvature.
 The equation \eqref{Elev} is called the level-set mean curvature flow equation.
 In this paper, we simply call it the level-set equation.
 The equation \eqref{Elev} is degenerate in the direction orthogonal to each level-set since each level set moves independently from other level sets.
 One needs a weak notion of a solution.
 The notion of viscosity solutions fits well.
 It turns out that there exists a unique global-in-time viscosity solution $u$ for any bounded uniformly continuous initial data $u_0$, i.e., $u_0\in BUC(\mathbb{R}^n)$ with the property that $u\in BUC\left(\mathbb{R}^n\times[0,T]\right)$ for any $T>0$;
 see e.g,\ \cite{G}.
 For given $\Gamma_0\subset\mathbb{R}^n$, we take $u_0\in BUC(\mathbb{R}^n)$ such that
\begin{equation} \label{EGin}
	\Gamma_0 = \left\{ x \in \mathbb{R}^n \bigm| u_0(x) = 0 \right\}.
\end{equation}
Such $u_0$ always exists by taking $u_0(x)=\operatorname{dist}(x,\Gamma_0)\wedge\delta$ for $\delta>0$, where $a\wedge b=\min(a,b)$.
 Our generalized solution to the mean curvature flow equation $V=H$ is given as
\begin{equation} \label{EGlev}
	\Gamma_t = \left\{ x \in \mathbb{R}^n \bigm| u(x,t) = 0 \right\}, \quad
	t \geq 0,
\end{equation}
where $u$ is the viscosity solution of \eqref{Elev} with $\left.u\right|_{t=0}=u_0$.
 Fortunately, the set $\Gamma_t$ only depends on $\Gamma_0$.
 It is even independent of the choice of the orientation $\nu$ since the mean curvature flow is orientation free;
 see e.g.\ \cite{ES} and \cite{G}.
 We may take $u_0\geq0$ so that $u\geq0$.
 We call $\{\Gamma_t\}$ as the level-set flow starting from $\Gamma_0$ or with initial data $\Gamma_0$.

To handle the prescribed boundary $\Sigma$, it turns out that it is reasonable to consider an obstacle problem for \eqref{Elev}.
 We consider \eqref{Elev} with constraint
\begin{equation} \label{Eob}
	\psi^- \leq u \leq \psi^+ \quad\text{in}\quad \mathbb{R}^n \times (0,\infty), 
\end{equation}
where $\psi^\pm$ is uniformly continuous in $\mathbb{R}^n\times[0,T]$ for any $T>0$ such that
\begin{equation} \label{Eb}
	\left\{ x \in \mathbb{R}^n \bigm| \psi^\pm(x,t) = 0 \right\} = \Sigma.
\end{equation}
The functions $\psi^+$ is called an upper obstacle and $\psi^-$ is called a lower obstacle.
 Such $\psi^\pm$ exists for example by taking $\psi^\pm=\pm\operatorname{dist}(x,\Sigma)$.
 Such an obstacle problem is studied by G.~Mercier \cite{M}.
 It is known in \cite{M} that there exists a unique global-in-time viscosity solution $u$ to \eqref{Elev} with \eqref{Eob} and the initial condition $u(\cdot ,0)=u_0$ on $\mathbb{R}^n$ satisfying 
 \[
	\psi^-(\cdot,0) \leq u_0\leq\psi^+(\cdot,0).
\]
 The resulting level-set flow only depends on $\Gamma_0$ and $\Sigma$ and independent of the choice of $\psi^\pm$ and $u_0$.
 Since \eqref{Elev} is orientation-free, we may assume $u_0\geq0$ so that $u\geq0$ and a lower obstacle $\psi^-$ is unnecessary.
 Thus, we consider \eqref{Elev} with
\[
	0 \leq u \leq \psi^+ \quad\text{in}\quad \mathbb{R}^n \times (0,\infty)
	\quad\text{and}\quad 0 \leq u_0 \leq \psi^+(\cdot,0) \quad\text{in}\quad \mathbb{R}^n
\]
under \eqref{Eb} and \eqref{EGin}.
 We call $\Gamma_t$ defined by \eqref{EGlev} through the viscosity solution of the level-set equation with obstacle $\psi^+$ the level-set flow with obstacle $\Sigma$ and initial data $\Gamma_0$.
 The level-set flow with obstacle $\Sigma$ and initial data $\Gamma_0$ can be obtained for arbitrary closed set $\Sigma$ and $\Gamma_0$ in $\mathbb{R}^n$.
 The set $\Sigma$ can be a single point or can have an interior.
 We propose that this is a generalized solution to a mean curvature equation \eqref{EMCD} with a prescribed boundary $\Sigma$ when $\Sigma$ is a codimension two manifold.

Our main goal in this paper is to prove that our level-set flow with obstacle $\Sigma$ is consistent with classical solution to \eqref{EMCD} and a level-set flow constructed by P.~Sternberg and W.~P.~Ziemer \cite{SZ} when $\Sigma$ is on the boundary $\partial U$ of strictly mean-convex bounded $C^2$ domain $U$ in $\mathbb{R}^n$, i.e., the inward mean curvature of $\partial U$ is positive.
 These simple looking problems turn to be nontrivial.

More precisely, let us consider
\begin{equation} \label{eq:CD}
    \left\{
    \begin{alignedat}{2} 
	&v_t -\operatorname{div} \left(\frac{Dv}{|Dv|}\right) |Dv| = 0 && \quad\text{in}\quad U \times(0,\infty) \\
	&v(\cdot,0) = v_0 && \quad\text{on}\quad \overline{U} \\
	&v(\cdot,t) = g && \quad\text{on}\quad \partial U\times(0,\infty),
    \end{alignedat}
    \right.
\end{equation}
where $U$ is a strictly mean-convex bounded $C^2$ domain in $\mathbb{R}^n$ satisfying $\Sigma\subset\partial U$.
 Here, $v_0\in C(\overline{U})$ and $g\in C(\partial U)$ are functions satisfying $\left.v_0\right|_{\partial U}=g$, and
\[
	\Sigma = \left\{ x \in \partial U \bigm| g(x) = 0 \right\}.
\]
In \cite{SZ}, the unique existence of global-in-time viscosity solutions $v\in C\left(\overline{U}\times[0,\infty)\right)$ to \eqref{eq:CD} is well established.
 It should be emphasized that the viscosity solution to \eqref{eq:CD} established in \cite{SZ} satisfies the boundary condition in the classical sense (not in the sense of viscosity solutions).
 In this paper, we establish
\begin{thm} \label{Tcsz}
Let $\Sigma$ and $\Gamma_0$ be compact sets in $\mathbb{R}^n$ with $\Sigma\subset\Gamma_0$.
 Assume that there is a bounded $C^2$ domain $U$ with strictly mean-convex boudary $\partial U$ containing $\Sigma$ and $\Gamma_0\backslash\Sigma\subset U$.
 Let $v$ be the viscosity solution in $U\times(0,T)$ of \eqref{eq:CD} 
 with $v_0=\operatorname{dist}(x,\Gamma_0)$ and $g=\left.v_0\right|_{\partial U}$, 
 and set 
 \[
 \Gamma_t^U:=\{x\in\mathbb{R}^n \mid v(x,t)=0\}. 
 \]
 Let $\Gamma_t$ be the level-set mean curvature flow with obstacle $\Sigma$ starting from $\Gamma_0$.
 Then $\Gamma_t^U=\Gamma_t$ for $0\leq t<T$.
 In particular $\Gamma_t^U$ is independent of $U$.
\end{thm}
Here, we briefly explain our strategy to prove Theorem \ref{Tcsz}.
 First let us take $v_0\in C(\overline{U})$ and $g\in C(\partial U)$ so that $v_0(x):=\operatorname{dist}(x,\Gamma_0)$ for $x\in\overline{U}$, and $g=0$ on $\Sigma$ and $g>0$ outside of $\Sigma$. 
 Let $v$ be the viscosity solution to \eqref{eq:CD}.
 Define an upper obstacle $\tilde{\psi}^+$ by
\[
	\tilde{\psi}^+(x,t) := v(x,t) + \sigma(x)
	\quad\text{for}\quad x \in \overline{U}, 
	\quad t \geq 0,
\]
where $\sigma\in C(\overline{U})$ is a give function satisfying $\sigma=0$ on $\partial U$ and $\sigma>0$ in $U$.
 We extend $\tilde{\psi}^+$ outside $\overline{U}$ so that the extended one $\psi^+$ is positive outside $\overline{U}$ and $\psi^+\in BUC\left(\mathbb{R}^n\times[0,\infty)\right)$.
 We note that $\left\{x \bigm|\psi^+(x,t)=0\right\}=\Sigma$ for all $t\geq0$.
 Let $u$ be the viscosity solution of \eqref{Elev} with obstacle $\psi^+$ and initial data $u_0$.
 Let $\Gamma_t$ be the level-set flow with obstacle $\Sigma$ starting from $\Gamma_0$.
 Since $U$ is strictly mean-convex, we observe that $\operatorname{dist}(x,\overline{U})$ is a viscosity subsolution of the obstacle problem near $\partial U$.
 By simple comparison, $u\geq\operatorname{dist}(x,\overline{U})$ so that $\Gamma_t\backslash\overline{U}=\emptyset$.
 One is able to construct a little bit smaller strictly mean-convex $C^2$ domain $V$ such that $\overline{V}$ is contained in $U$ except near $\Sigma$ and $\Sigma\subset\partial V$.
 By comparison with $\operatorname{dist}(x,\overline{V})$, we even conclude that $u>0$ on $\partial U\backslash\Sigma$.
 By comparison, we conclude that $u\leq v$ in $U$, which implies $\Gamma_t^U\subset\Gamma_t$ and that $u$ is a viscosity solution of \eqref{Elev} in $U$ with no obstacle.
 Since $u>0$ on $\partial U\backslash\Sigma$, we are able to renormalize the value of $u$ on $\partial U$ so that $v\leq\theta\circ u$ on $\partial U$, where $\theta\in C[0,\infty)$ is some increasing function such that $\theta(0)=0$.
 By the invariance $\theta\circ u$ is a viscosity solution to \eqref{Elev} in $U$ (with no obstacle).
 By comparison $v\leq\theta\circ u$ in $U$, which implies that $\Gamma_t\subset\Gamma_t^U$.
 We thus conclude that $\Gamma_t=\Gamma_t^U$.
 We give the detail of the proof in Section \ref{S3}. 

Next, we prove the consistency with classical solutions along the line of \cite{ES} (see also \cite{GG1}) with extra care near $\Sigma$.
\begin{thm} \label{Tcons} 
Let $\{\Gamma_t^s\}_{0\leq t \leq T}$ be a continuous family (in the sense of Hausdorff distance) of compact evolving hypersurfaces in $\mathbb{R}^n$ that contains a $k$ codimensional ($n-k$ dimensional) $C^2$ submanifold $\Sigma$ (independent of $t$ possibly having a $C^2$ geometric boundary) with $k\geq2$.
 Assume that $\{\Gamma_t^s\}_{0\leq t \leq T}$ is $C^{2,1}$ outside $\Sigma$ and satisfies the mean curvature flow equation $V=H$ outside $\Sigma$ with initial data $\left.\Gamma_t^s\right|_{t=0}=\Gamma_0$, which is $C^2$ outside $\Sigma$.
 Let $\{\Gamma_t\}_{t\geq0}$ be a level-set mean curvature flow with obstacle $\Sigma$.
 If $\left.\Gamma_t\right|_{t=0}=\Gamma_0$, then $\Gamma_t^s=\Gamma_t$ for $0\leq t\leq T$.
\end{thm}
By $C^{2,1}$, we mean that $\Gamma_t^s$ is $C^2$ in space and $C^1$ in time $t$;
 for precise definition, see \cite[Chapter 1]{G}.
 Theorem \ref{Tcons} evidently contains the problem \eqref{EMCD} when $\Sigma$ is a $2$-codimensional submanifold. 
 It also applies the case when $\Sigma$ has a boundary of lower dimension;
 see Figure \ref{Fhcod}.
\begin{figure}[htb]
\centering
\includegraphics[width=5cm]{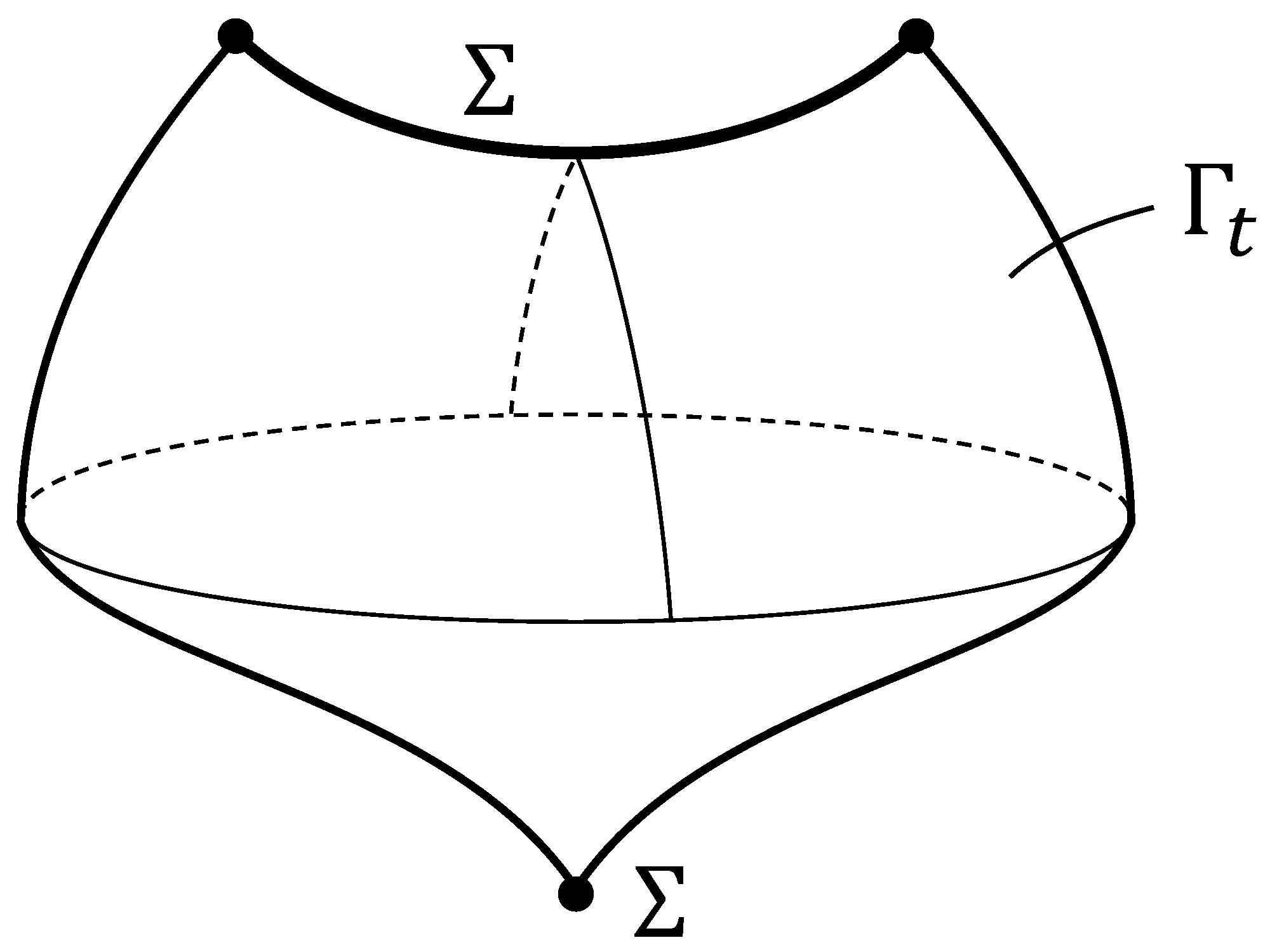}
\caption{example of $\Sigma$} \label{Fhcod}
\end{figure}

This paper is organized as follows.
 In Section \ref{S2}, we review a few basic facts on obstacle problems for a general spatially homogeneous equations including level-set flow equations. 
 In particular, we prove the spatially Lipschitz property and time H\"older continuity of a viscosity solution of the level-set flow equation provided that the initial data and obstacle is Lipschitz continuous.
 This improves the results in \cite{GTZ}.
 In Section \ref{S3}, we compare the level-set flow with obstacle $\Sigma$ to the level-set flow constructed by \cite{SZ} when $\Gamma_0$ is $\overline{U}$, where $U$ is a strictly mean-convex bounded $C^2$ domain.
 In particular, we prove Theorem \ref{Tcsz}. 
 In Section \ref{S4}, we compare with our solution to smooth solutions.
 In particular, we prove Theorem \ref{Tcons}. 
 In Section \ref{S5}, we give a few perspectives for further study.

The second author is grateful to Professor Tatsuya Miura for pointing out  \cite{Amb, AM} related to Lemma \ref{Lsigma}.

\section{Basic properties for obstacle problems} \label{S2} 

In this section as in \cite{M}, we recall several basic results for obstacle problems which applies to our level-set equation.
 Instead of stating just for \eqref{Elev}, we shall state them for general equations.

\subsection{Obstacle problems} \label{SS21} 

We consider an evolution equation of the form
\begin{equation} \label{EF1}
	u_t + F(\nabla u, \nabla^2 u) = 0.
\end{equation}
The equation \eqref{Elev} corresponds to the case
\begin{equation} \label{EF2}
	F(p,X) = -\operatorname{trace}\left( \left(I-\frac{p\otimes p}{|p|^2} \right) X \right).
\end{equation}
The standard set of assumptions for $F$ (see \cite{CGG, G}) including this example is
\begin{itemize}
\item[(F1)] $F:(\mathbb{R}^n\backslash 0)\times\mathbb{S}^n\to\mathbb{R}$ is continuous;
\item[(F2)] (degenerate ellipticity) $F(p,X)\leq F(p,Y)$ if $X\geq Y$, $p\in\mathbb{R}^n\backslash\{0\}$; 
\item[(F3)] $-\infty< F_*(0,O)=F^*(0,O)<\infty$.
\end{itemize}
Here $\mathbb{S}^n$ denotes the space of all real symmetric matrices and $X\geq Y$ means that $X-Y$ is a nonnegative definite matrix.
 The function $F^*$, $F_*$ are upper and lower semicontinuous envelope of $F$, respectively.
 Namely,
\begin{gather*}
	F^*(p,X) = \limsup_{q\to p,Y\to X} F(q,Y), \\
	F_*(p,X) = \liminf_{q\to p,Y\to X} F(q,Y).
\end{gather*}
We consider \eqref{EF1} under constraint $\psi^-\leq u\leq\psi^+$.
 The definition of viscosity solutions for bilateral obstacle problems is rather standard, see e.g.\ \cite{Ya} but we give it for completeness.
\begin{definition} \label{Dgol}
Let $\Omega$ be an open set in $\mathbb{R}^n$ and $T>0$ and $Q=\Omega\times(0,T)$.
 Let $\psi^\pm$ be a continuous function satisfying $\psi^-\leq\psi^+$ in $\overline{Q}$.
\begin{itemize}
\item[{\rm(i)}] A function $u:Q\to\mathbb{R}\cup\{-\infty\}$ is a viscosity \emph{subsolution} of \eqref{EF1} with obstacles $\psi^\pm$ in $Q$ if
\begin{itemize}
\item[{\rm(a)}] $u^*(x,t)<\infty$ for $(x,t)\in\overline{Q}$;
\item[{\rm(b)}] $\varphi_t(\hat{x},\hat{t})+F_*\left(\nabla\varphi(\hat{x},\hat{t}), \nabla^2\varphi(\hat{x},\hat{t})\right) \leq0$ whenever $\left(\varphi,(\hat{x},\hat{t})\right)\in C^2(Q)\times Q$ satisfies
\begin{gather*}
	\max_Q (u^*-\varphi) = (u^*-\varphi)(\hat{x},\hat{t}) \\
	\text{and}\quad u^*(\hat{x},\hat{t}) > \psi^- (\hat{x},\hat{t});
\end{gather*}
\item[{\rm(c)}] $\psi^-\leq u^*\leq\psi^+$ in $\overline{\Omega}\times[0,T)$.
\end{itemize}
\item[{\rm(i\hspace{-1.5pt}i)}] A function $u:Q\to\mathbb{R}\cup\{+\infty\}$ is a viscosity \emph{supersolution} of \eqref{EF1} with obstacles $\psi^\pm$ in $Q$ if
\begin{itemize}
\item[{\rm(a)}] $u_*(x,t)>\infty$ for $(x,t)\in\overline{Q}$; 
\item[{\rm(b)}] $\varphi_t(\hat{x},\hat{t})+F^*\left(\nabla\varphi(\hat{x},\hat{t}), \nabla^2\varphi(\hat{x},\hat{t})\right) \geq0$ whenever $\left(\varphi,(\hat{x},\hat{t})\right)\in C^2(Q)\times Q$ satisfies
\begin{gather*}
	\min_Q (u_*-\varphi) = (u_*-\varphi)(\hat{x},\hat{t}) \\
	\text{and}\quad u_*(\hat{x},\hat{t}) < \psi^+ (\hat{x},\hat{t});
\end{gather*}
\item[{\rm(c)}] $\psi^-\leq u_*\leq\psi^+$ in $\overline{\Omega}\times[0,T)$.
\end{itemize}
\item[{\rm(i\hspace{-1.5pt}i\hspace{-1.5pt}i)}] 
If $u$ is simultaneously a viscosity sub- and supersolution of \eqref{EF1} with obstacles $\psi^\pm$ in $Q$, 
then $u$ is said to be a viscosity solution. 
\end{itemize}
\end{definition}

The function $\psi^+$ (resp., \ $\psi^-$) is called an upper (resp., a lower) obstacle.
 Our definition trivially extends to the problem with upper obstacle $\psi^+$ only by taking $\psi^-\equiv-\infty$.
 Our definition is slightly different from \cite{M} since we do not include initial data in our definition.
 The definitions in \cite{ES} and \cite{CGG} (without $\psi^\pm$) given for $F$ in \eqref{EF2} are different each other at the place where $\nabla\varphi(\hat{x},\hat{t})=0$ but it turns out that they are equivalent as proved by \cite{BG};
 see also \cite[Proposition 2.2.8]{G} for general $F$ satisfying (F1) -- (F3).

As already pointed out in \cite{M} (see also \cite[Example 1.7]{CIL}), $u$ is a viscosity subsolution of \eqref{EF1} with obstacles $\psi^\pm$ if and only if $u$ is a viscosity subsolution of
\[
	G (x, t, u, u_t, \nabla u, \nabla^2 u ) = 0
\]
with
\[
	G (x, t, r, \tau, p, X) = \max \left( \min \left( \tau + F(p,X), r - \psi^-(x,t)\right), r - \psi^+(x,t) \right).
\]
This approach is useful to discuss stability of viscosity solutions, but to establish comparison principle, the evolutional structure like $u_t+F(\nabla u,\nabla^2 u)=0$ is a key.

Here is a fundamental unique existence result.
\begin{prop} \label{Pue}
Assume (F1) -- (F3).
 Assume that $\psi^\pm$ be uniformly continuous in $\mathbb{R}^n\times[0,T]$ for any $T>0$ and $\psi^-\leq\psi^+$.
 Assume that $u_0\in BUC(\mathbb{R}^n)$ with $\psi^-\leq u_0\leq\psi^+$ in $\mathbb{R}^n\times\{0\}$.
 Then there exists a unique viscosity solution $u$ to \eqref{EF1} with obstacles $\psi^\pm$ in $\mathbb{R}^n\times(0,T)$ such that $u(x,0)=u_0(x)$ and $u\in BUC\left(\mathbb{R}^n\times[0,T]\right)$.
\end{prop}
This is stated in \cite[Theorem 1]{M} under additional assumption that $\psi^\pm$ are bounded and that $F$ is assumed to be geometric in the sense of \cite{CGG}, i.e.,
\begin{equation} \label{Ege}
	F(\lambda p, \lambda X+\sigma p\otimes p)
	= \lambda F(p,X) \quad\text{for all}\quad
	\lambda > 0,\ \sigma \in \mathbb{R}, \ p \in \mathbb{R}^n \backslash\{0\},\ X \in \mathbb{S}^n
\end{equation}
However, it is unnecessary.
 The basic strategy is by now standard.
 We first establish a comparison principle as in \cite{M}, which is an adaptation of the standard strategy \cite[Theorem 8.2]{CIL} and a technique found in \cite[Theorem 4.1]{GGIS}, where the comparison principle is established for general equations including \eqref{EF1} but without obstacles;
 see \cite[Theorem 3.1.4]{G}.
 Once the comparison principle has been established, the existence of a viscosity solution can be established by what is called Perron's method by constructing a barrier near initial data.
 The uniqueness follows from the comparison principle.
 We give here a simple version of the comparison principle  for uniformly continuous functions in $\mathbb{R}^n$ for later convenience.
 In the next section, we give a version in a bounded domain but for semicontinuous functions.
\begin{prop} \label{Pcom}
Assume that (F1) -- (F3).
 For $T>0$, let $u$ and $v$, respectively, be a viscosity sub- and supersolution of \eqref{EF1} in $\mathbb{R}^n\times(0,T)$ with obstacles $\psi^\pm$ such that both $u$ and $v$ are uniformly continuous in $\mathbb{R}^n\times[0,T]$.
 If $u\leq v$ at $t=0$, then $u\leq v$ in $\mathbb{R}^n\times(0,T]$. 
\end{prop}
For the proof see \cite{GGIS} or \cite[Proposition 1]{M}.

We conclude this subsection by proving regularity of the viscosity solution when both $\psi^\pm$ and $u_0$ are Lipschitz. 
 In \cite{M}, a spatial uniform regularity of a viscosity solution is estimated as
\[
	\left| u(x,t) - u(y,t) \right|
	\leq \omega \left( e^{Lt}|x-y| \right), 
\]
where $\omega$ is a modulus of continuity of the initial data $u_0$ while $L$ is some constant by extending the result of \cite[Lemma 2.15]{Fo}.
 We shall prove the Lipschitz preserving property which is an extension of results without obstacles (\cite{GGIS}, \cite[Chapter 3.5]{G}).
\begin{thm} \label{TLip}
Assume that (F1) -- (F3).
 Let $u$ be the viscosity solution of \eqref{EF1} with obstacles $\psi^\pm$ and initial data $u_0$.
 Assume that $u_0$, $\psi^\pm$ are Lipschitz continuous with constant $L$.
 Then
\begin{itemize}
\item[(i)] $u$ is Lipschitz in space with constant $L$, i.e.,
\begin{equation} \label{ELI}
	\left| u(x,t) - u(y,t) \right| \leq L |x-y|
	\quad\text{for}\quad x,y \in \mathbb{R}^n, \quad t > 0.\end{equation}  
Moreover,
\item[(i\hspace{-1.5pt}i)] $u$ is $1/(1+\alpha)$-H\"older continuous in time, i.e., 
\begin{equation} \label{EHOE}
	\left| u(x,t) - u(x,s) \right| \leq C |t-s|^{1/(1+\alpha)}
	\quad\text{for}\quad t,s \geq 0,
	\quad x \in \mathbb{R}^n
\end{equation}
with some constant $C$ (independent of $t$, $s$ and $x$) provided that $\alpha>0$ satisfies
\begin{equation} \label{EGr}
	c_M := \sup_{|p|\leq M,X\in\mathbb{S}^n} \left| F(p,X) \right| \bigm/ \left(|X| + 1\right)^\alpha < \infty
\end{equation}
for any $M>0$ and that $\psi$ is independent of time.
 In particular, when $F$ is given in \eqref{EF2} so that $\alpha=1$, $u$ is $1/2$-H\"older continuous in time.
\end{itemize}
\end{thm}
\begin{remark} \label{RLip}
A similar Lipschitz continuity of $u$ is proved in \cite{GTZ}
\begin{equation} \label{Edr}
	F(p,X) = -\operatorname{trace} \left( \left(I-\frac{p\otimes p}{|p|^2} \right) X \right) + k|p|, \quad
	k \in \mathbb{R}
\end{equation}
under additional regularity assumptions on the initial data.
 Time regularity was not mentioned.
 Our proof is a modification of the proof given in \cite{GP} for spatially inhomogeneous  crystalline mean curvature flow corresponding to the case $\alpha=1$.
\end{remark}
\begin{proof}
We first prove the Lipschitz continuity.
 We set
\begin{gather*}
	u^+(x,t) := \left(u(x+z,t) + L|z|\right) \wedge \psi^+(x,t), \\
	u^-(x,t) := \left(u(x+z,t) - L|z|\right) \vee \psi^-(x,t)
\end{gather*}
for $x, z\in \mathbb{R}^n$ and $t\in [0,T)$. 
 Here, $a\wedge b=\min(a,b)$, $a\vee b=\max(a,b)$ for $a,b\in\mathbb{R}$. 
 We shall claim $u^+$ is a viscosity supersolution and $u^-$ is a viscosity subsolution to the obstacle problem, respectively, satisfying $u^+\geq u_0\geq u^-$ at $t=0$.

Here, we only prove the claim for $u^+$.
 One hand, by definition $u^+\leq \psi^+$, and we have
\[
u(x+z,t)+L|z|\geq\psi^-(x+z,t)+L|z|\geq\psi^-(x,t) 
\]
by using the fact that $\psi^-$ is $L$-Lipschiz. It means 
\begin{equation} \label{EL1}
	\psi^-(x,t)\leq u^+(x,t)\leq\psi^+(x,t) 
\end{equation}
for all $(x,t)\in \mathbb{R}^n\times[0,T)$. On the other hand, since the initial data $u_0$ is $L$-Lipschiz, we arrive at
\begin{equation} \label{EL2}
u(x+z,0)+L|z|=u_0(x+z)+L|z|\geq u_0(x),
\end{equation}
which implies $u^+(x,0)\geq u_0(x)$. Next, we prove $u^+$ satisfies \eqref{EF1}.
 Let $\varphi$ be a smooth test function,  $(\hat x,\hat t)$ be a minimum point, such that
\[
	u^+(\hat x,\hat t)-\varphi(\hat x,\hat t)\leq u^+(x,t)-\varphi(x,t),
\]
and assume 
$u^{+}(\hat x,\hat t)<\psi^+(\hat x,\hat t)$. 
By definition of $u^+$, $u^{+}(\hat x,\hat t)=u(\hat x+z,\hat t)+L|z|$, which implies 
We have
\begin{equation} \label{EL3}
	u(\hat x+z,\hat t)+L|z|-\varphi(\hat x,\hat t)\leq u(x+z,t)+L|z|-\varphi(x,t).
\end{equation}
Now, we set $y=x+z$, $\hat y=\hat x+z$, and define $\varphi(x,t)=\phi(x+z,t)$, then \eqref{EL3} becomes 
\begin{equation} \label{EL4}
	u(\hat y,\hat t)-\phi(\hat y,\hat t)\leq u(y,t)-\phi(y,t).
\end{equation}
Thus $u-\phi$ attains its minimum at $(\hat y,\hat t)$. We also have 
\[
	u(\hat y,\hat t)=u(\hat x +z,\hat t)< \psi^+(\hat x,\hat t)-L|z|\leq\psi^+(\hat x+z,\hat t)=\psi^+(\hat y,\hat t).
\]
As $u$ is a viscosity supersolution at $(\hat y,\hat t)$, we deduce that
\[
	\left(\phi_t +  F^*(\nabla\phi,\nabla^2\phi)\right)(\hat y,\hat t) \geq 0.
\]
Besides, since $\phi_t(\hat y,\hat t)=\varphi_t(\hat x,\hat t)$, $\nabla\phi(\hat y,\hat t)=\nabla\varphi(\hat x,\hat t)$, we obtain
\[
	\left(\phi_t +  F^*(\nabla\phi,\nabla^2\phi)\right)(\hat x,\hat t) \geq 0.
\]
Consequently, $u^+$ is a viscosity supersolution. Similarly, we can prove $u^-$ is a viscosity subsolution. By the comparison principle (Proposition \ref{Pcom}), $u^-\leq u\leq u^+$, which means
\[
	u(x+z,t)-L|z|\leq u(x,t)\leq u(x+z,t)+L|z|.
\]
Therefore,
\[
\left|u(x+z,t)-u(x,t)\right|\leq L|z|.
\]

We next prove the H\"older continuity in time.
 By Lipschitz continuity \eqref{ELI}, we have
\[
	u(x,t) - u(x_0,t) \leq L|x-x_0| \leq L\left(\delta+|x-x_0|^2\right)^{1/2} \quad\text{for}\quad \delta > 0.
\]
We set
\begin{align*}
	&v^+(x,t) := h(x,t) \wedge \psi^+(x,t), \\
	&h(x,t) := L\left(ct\delta^{-\alpha/2}+\left(\delta+|x-x_0|^2\right)^{1/2} \right) + u_0(x_0).
\end{align*}
We claim that $v^+$ is a viscosity supersolution of \eqref{EF1} with obstacles $\psi^\pm$ provided that $c$ is taken sufficiently large (independent of $\delta<1$).
 In the place where $v^+(x,t)=\psi^+(x,t)$, nothing has to be done.
 We may assume that $v^+(x,t)=h(x,t)$.
 Since $|\nabla h|\leq L$, $|\nabla^2h|\leq L/\delta^{1/2}$, the assumption \eqref{EGr} implies that
\[
	\left| F^*(\nabla h,\nabla^2h) \right|
	\leq c_L\left( \frac{L}{\delta^{1/2}} + 1 \right)^\alpha
	\leq C_L' \delta^{-\alpha/2}
\]
with some $C_L'$ independent of $\delta>0$.
 Thus, if $c>C_L'/L$, then
\[
	h_t + F^*(\nabla h,\nabla^2h) \geq Lc\delta^{-\alpha/2} - C_L' \delta^{-\alpha/2} \geq 0.
\]
Since $u_0$ is Lipschitz in space and satisfies $u_0\geq\psi^-$, we observe that  
\[
	h(x,t) \geq Lct\delta^{-\alpha/2} + u_0(x) \geq \psi^-(x).
\]
Thus, $\psi^-\leq v^+\leq\psi^+$ so $v^+$ is a viscosity supersolution.
 Since $v^+(x,0)\geq u_0(x)$, by the comparison principle (Proposition \ref{Pcom}), we arrive at
\[
	u(x_0,t) \leq v^+(x_0,t) \leq h(x_0,t) \leq u_0(x_0) + L(ct\delta^{-\alpha/2} + \delta^{1/2})
\]
for any $x_0\in\mathbb{R}^n$.
 We thus obtain
\[
	u(x,t) \leq u_0(x) + L(ct\delta^{-\alpha/2} + \delta^{1/2}).
\]
A symmetric argument implies that
\[
	u(x,t) \geq u_0(x) - L(ct\delta^{-\alpha/2} + \delta^{1/2}).
\]
We take $\delta$ so that $ct\delta^{-\alpha/2}=\delta^{1/2}$ to obtain
\[
	\left| u(x,s) - u(x,0) \right| \leq As^{1/(1+\alpha)}
	\quad\text{for}\quad s > 0
\]
with $A=2Lc^{1/(1+\alpha)}$.
 If one considers $u(x,t)$ as a initial data, we conclude that
\[
	\left| u(x,t+s) - u(x,t) \right| \leq As^{1/(1+\alpha)}
	\quad\text{for}\quad s \geq 0
\]
and this is the desired estimate.
\end{proof}
\begin{remark} \label{Rbar}
The function $v^\pm$ is often called a barrier which prevents a sudden jump of a viscosity solution at $t=0$.
 A standard choice is
\[
	h(x,t) = L\left(\frac{ct}{\delta} + \delta + \frac{|x-x_0|^2}{4\delta}\right) + u_0(x_0)
\]  
by observing $|x|\leq\delta+|x-x_0|^2/4\delta$.
 This choice is not convenient if $F(p,X)$ is of the form \eqref{Edr} with $k\neq0$ since $|\nabla h|$ is unbounded and $|\nabla h|\to\infty$ as $\delta\downarrow0$.
\end{remark}
\subsection{Orientation free motion} \label{SS22} 

We consider a general surface evolution
\begin{equation} \label{EGEN}
	V = f (\nu, \mathbf{A})
\end{equation}
for an evolving hypersurface $\{\Gamma_t\}$, where $\mathbf{A}$ denotes the second fundamental form in the direction of $\nu$.
 If $f(\nu,A)=\operatorname{trace}A$, the equation becomes the mean
 curvature flow equation $V=H$.
 If each level set of a viscosity solution of \eqref{EF1} satisfies \eqref{EGEN}, i.e., \eqref{EF1} is the \emph{level-set equation} of \eqref{EGEN}, then the equation is geometric, i.e., $F$ satisfies \eqref{Ege}.
 Conversely, if \eqref{EF1} satisfies \eqref{Ege}, then it is a level-set equation of \eqref{EGEN} provided that $F$ satisfies (F1) and (F2);
 see \cite{GG} or \cite[Theorem 1.6.12]{G}.
 Without (F2), \eqref{EF1} with \eqref{Ege} may not be a level-set equation of any \eqref{EGEN};
 \cite{GG} or \cite[Subsection 1.6.4]{G}.

The assumption (F3) corresponds the growth condition of \eqref{EGEN} of $f$ with respect to $\mathbf{A}$.
 It is fulfilled if the growth is sublinear or linear in $\mathbf{A}$;
 see \cite[Subsection 1.6.5]{G}.

The evolution of each level set is determined by the corresponding initial level set and the level set of an obstacle function.
\begin{prop} \label{PUN}
Assume the same hypotheses of Proposition {\rm\ref{Pue}} concerning $F$, $u_0$ and $\psi^\pm$.
 Let 
\[
	S_\pm = \left\{ (x,t) \in \mathbb{R}^n \times [0,T] \bigm|
	\psi^\pm(x,t) = 0 \right\}.
\] 
Let $u$ be the viscosity solution of \eqref{EF1} with obstacles $\psi^\pm$ in $\mathbb{R}^n\times(0,T)$ and initial data $u_0$.
 Then the set
\[
	D_t = \left\{ x \bigm| u(x,t) > 0 \right\}	\ \text{{\rm(}resp., }\ 
	E_t = \left\{ x \bigm| u(x,t) \geq 0 \right\}\text{{\rm)},} \quad
	t \in [0,T)
\]
is uniquely determined by $D_0$ {\rm(}resp., \ $E_0${\rm)} and $S^\pm$, and it is independent of the choice of $u_0$ and $\psi^\pm$.
\end{prop}
%
If the equation \eqref{EGEN} is orientation-free, i.e., in the corresponding equation \eqref{EF1}
\begin{equation} \label{Eorf}
	F(-p,-X) = -F(p,X), \quad
	p \in \mathbb{R}^n \backslash \{0\}, \quad
	X \in \mathbb{S}^n
\end{equation}
holds, then we are able to conclude a stronger statement.
\begin{prop} \label{Porf}
Assume the same hypotheses of Proposition \ref{PUN} concerning $u_0$, $\psi^\pm$ and $u$.
 Assume further \eqref{Eorf}.
 Then the set $\Gamma_t=\left\{ x \bigm|u(x,t) = 0 \right\}$ {\rm($t\geq0$)} uniquely determined by $\Gamma_0$ and $S=S_+\cup S_-$ and independent of the choice of $u_0$,  $\psi_\pm$.
 In particular, we may assume $u_0\geq0$ so that $u\geq0$ and $\psi_-=-\infty$, $\psi_+\geq0$.
\end{prop}
For orientation-free motion, we give a definition of a level-set flow with obstacle.
 For simplicity, we assume that the obstacle is standing, i.e., the cross-section of $S$
\[
	S(t) = \left\{ x \in \mathbb{R}^n \bigm| (x,t) \in S \right\}
\]
is independent of time.
 In this case, we write $S(t)=\Sigma\subset\mathbb{R}^n$.
\begin{definition} \label{Dlev}
Let $\Sigma$ and $\Gamma_0$ be closed sets in $\mathbb{R}^n$.
 Assume that \eqref{EGEN}  is orientation-free and its level-set equation \eqref{EF1} satisfies (F1) -- (F3).
 We say that a family of closed sets $\{\Gamma_t\}_{t\in(0,T)}$ is a \emph{level-set flow} of \eqref{EGEN} with obstacle $\Sigma$ and initial data $\Gamma_0$ if there exists a nonnegative viscosity solution $u\in BUC\left(\mathbb{R}^n\times[0,T)\right)$ of \eqref{EF1} in $\mathbb{R}^n\times(0,T)$ with upper obstacle $\psi^+\geq0$ such that
 \[
	\Gamma_t = \left\{ x \in \mathbb{R}^n \bigm| u(x,t) = 0 \right\}, \quad
	\Sigma = \left\{ x \in \mathbb{R}^n \bigm| \psi^+(x,t) = 0 \right\} 
	\quad\text{for}\quad t \in [0,T), 
\]
where $\psi^+$ is uniformly continuous in $\mathbb{R}^n\times[0,T]$.
\end{definition}
%
Since the mean curvature flow equation is orientation-free, Definition \ref{Dlev} is sufficient to handle a level-set mean curvature flow with obstacles.

Proposition \ref{Porf} is proved in \cite{M} at least for \eqref{EF1} with \eqref{EF2}.
 The basic strategy to prove both Propositions \ref{PUN} and \ref{Porf} is the same to the case without obstacles.
 Key ingredients are the invariance lemma (stated in the next subsection) which guarantees invariance under the change of dependent variables as well as comparison principle as in \cite{ES, CGG} at least $D_t$ and $E_t$ are bounded.
 The proof for unbounded case is more involved but it is still valid;
 see e.g.\ \cite[Theorem 4.2.11]{G}.
 If the equation \eqref{EGEN} is orientation-free, then $|u|$ is also a viscosity solution with obstacles $|\psi^+|$, $|\psi^-|$.
 Such observation was found in \cite{ES} for \eqref{EF1} with \eqref{EF2} with no obstacles;
 see \cite[Subsection 4.2.2]{G} for general orientation-free motion.

\subsection{Avoidance principle} \label{SS23}

As already pointed out in \cite{ES}, the Lipschitz preserving property (Theorem \ref{TLip}) implies the avoidance principle of two level-set flows stated below.
 Let $\operatorname{dist}(A,B)$ denote the distance of two sets $A$ and $B$, i.e.,
\[
	\operatorname{dist}(A,B) = \inf \left\{|x-y| \bigm| x \in A,\ y \in B \right\}.
\]
If $A$ is a singleton $\{x\}$, we simply write $\operatorname{dist}(x,B)$ instead of $\operatorname{dist}\left(\{x\},B\right)$.
\begin{cor} \label{Cav}
Assume the same hypothesis of Definition \ref{Dlev} concerning \eqref{EGEN}.
 For $T>0$,  let $\{\Gamma_t^i\}_{t\in(0,T)}$ be a level-set flow of \eqref{EGEN} with obstacle $\Sigma^i$ and initial data $\Gamma_0^i$, 
 where $\Sigma^i$ and $\Gamma_0^i$ with $\Sigma^i\subset\Gamma_0^i$ are closed sets in $\mathbb{R}^n$ for $i=1,2$.
 Then
\[
	\operatorname{dist}(\Gamma_t^1, \Gamma_t^2) \geq
	\operatorname{dist}(\Gamma_0^1, \Gamma_0^2)
	\quad\text{for}\quad t \in (0,T)
\]
provided that $\Sigma^2=\emptyset$ and $\inf\left\{\operatorname{dist}(\Gamma_t^2,\Sigma^1)\bigm|t\in(0,T)\right\}>0$.
\end{cor}
\begin{proof}
We may assume that $\operatorname{dist}(\Gamma_t^2,\Sigma^1)\geq1$ by dilation.
 We set
\begin{equation*}
	u_0(x)= \left \{
\begin{array}{ll}
	\displaystyle M\operatorname{dist}(x,\Gamma_0^1) \wedge \frac{d}{2}
	&\text{if}\quad \displaystyle\operatorname{dist}(x,\Gamma_0^2) \geq \frac{d}{2}, \vspace{0.4em} \\
	\displaystyle\left(d-\operatorname{dist}(x,\Gamma_0^2)\right)\vee \frac{d}{2}
	&\text{if}\quad \displaystyle\operatorname{dist}(x,\Gamma_0^2) < \frac{d}{2}, 
\end{array}
\right.
\end{equation*}
%
where $d=\operatorname{dist}(\Gamma_0^1,\Gamma_0^2)$ and $M>1$; see Figure \ref{FCoIn}. 
\begin{figure}[htb]
\centering
\includegraphics[width=5.5cm]{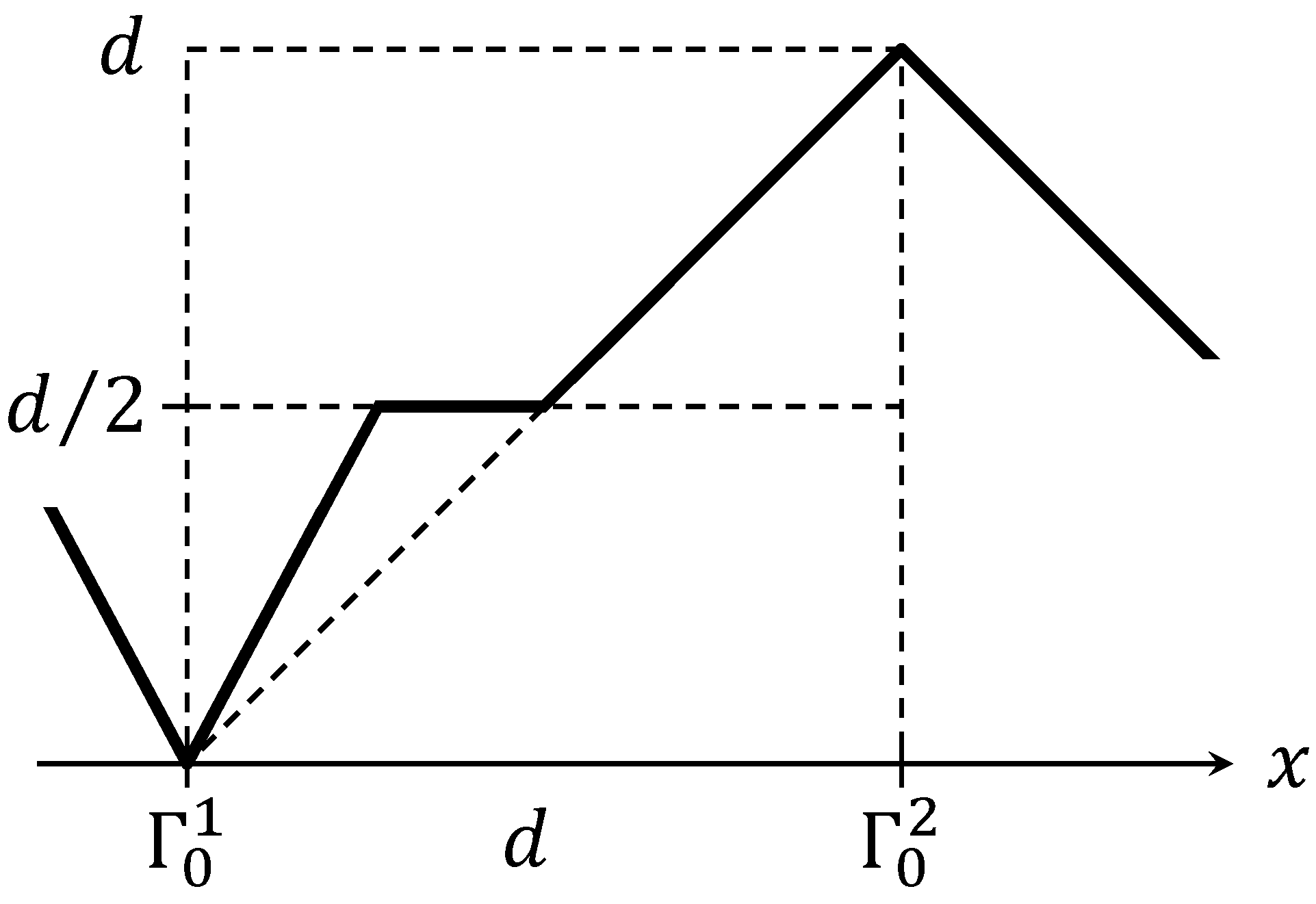}
\caption{graph of $u_0$} \label{FCoIn}
\end{figure}

 By definition, $u_0(x)$ is continuous on the set 
 \[
	\left\{x\in\mathbb{R}^n \bigm| \operatorname{dist}(x,\Gamma_0^2)=\frac{d}{2} \right\}.
\]
 Thus, $u_0$ is Lipschitz on $\mathbb{R}^n$ with constant $M$.
 By definition,
\[
	u_0(x) \leq M \operatorname{dist}(x,\Sigma^1) =: \psi^+(x).
\]
Let $u$ be the viscosity solution of \eqref{EF1} with obstacle $\psi^+$ and initial data $u_0$.
 By definition,
\[
	\Gamma_t^1 = \left\{x \in\mathbb{R}^n \bigm| u(x,t) = 0 \right\}.
\]
We have assumed that $ \operatorname{dist}(\Gamma_t^2,\Sigma^1)\geq1$ for $t\in(0,T)$.
 Thus, 
\[
	\Gamma_t^2 \cap \left\{ x \in \mathbb{R}^n \bigm|
	M \operatorname{dist}(x,\Sigma^1) \leq 1 \right\} = \emptyset.
\]
Since the equation is orientation free, $d-u$ solves \eqref{EF1} without obstacle, and
\[
	\Gamma_t^2 = \left\{x \in\mathbb{R}^n \bigm| u(x,t) = d \right\}.
\]
Since $u(\cdot,t)$ is spatially Lipschitz with constant $	M$ by Theorem \ref{TLip}, this implies that $\operatorname{dist}(\Gamma_t^2,\Gamma_t^1)\geq d/M$.
 Sending $M\downarrow1$, the proof of Corollary \ref{Cav} is now complete.
\end{proof}

As an application, one is able to construct an example of neck-pinching by using Angenent's self-similar shrinking  torus (doughnut) as observed in \cite{An};
 note that the shrinking doughnut does not touch $\Gamma_t$ by Corollary \ref{Cav}.
 See Figure \ref{Fsingfor}.

In Corollary \ref{Cav}, the assumption $\Sigma^2=\emptyset$ is necessary.
 See Figure \ref{FColl}.
\begin{figure}[htb]
\centering
	\begin{minipage}[t]{0.45\linewidth}
\centering
\includegraphics[width=7cm]{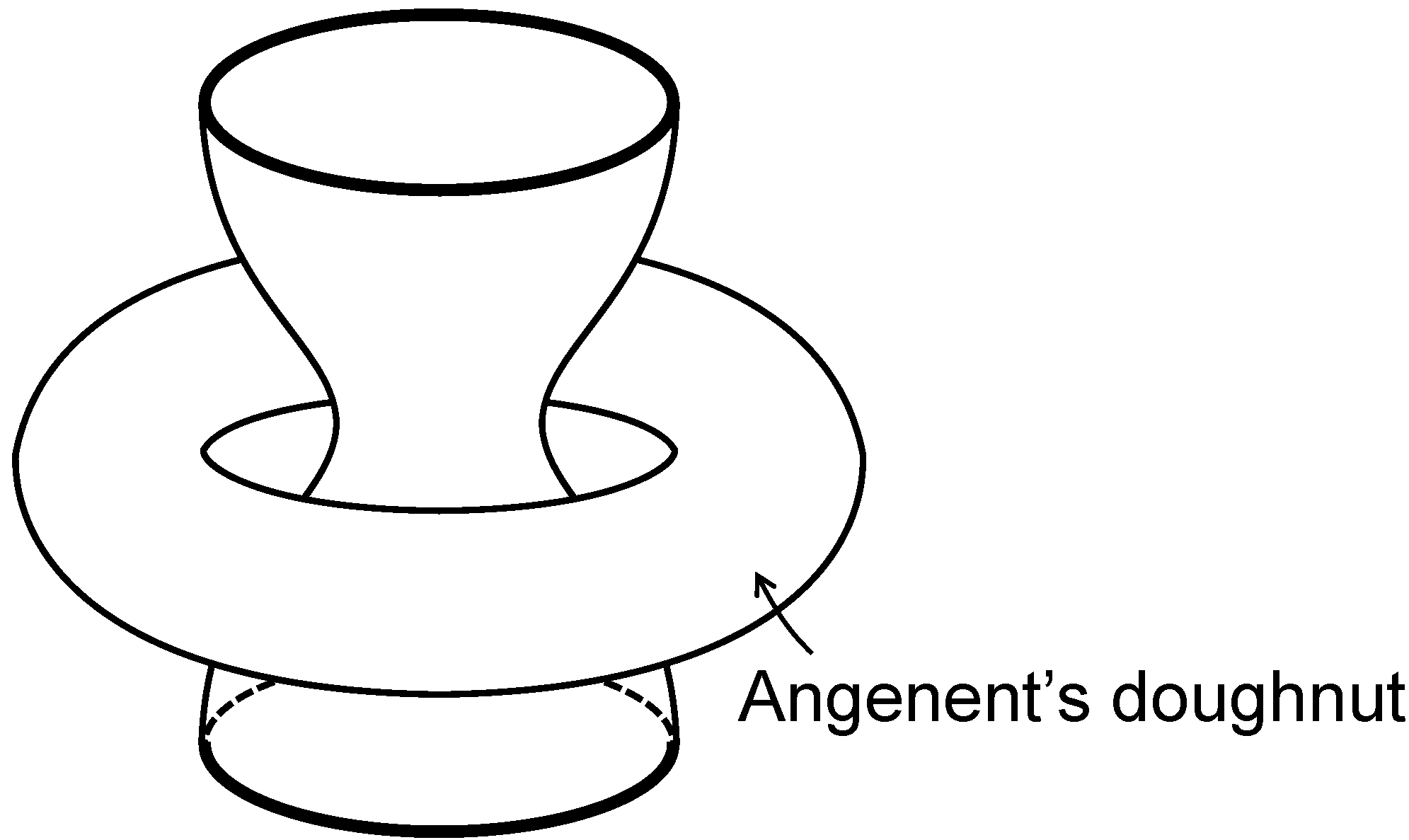}
\caption{example of pinching} \label{Fsingfor}
	\end{minipage}
	\begin{minipage}[t]{0.40\linewidth}
\centering
\includegraphics[width=5.5cm]{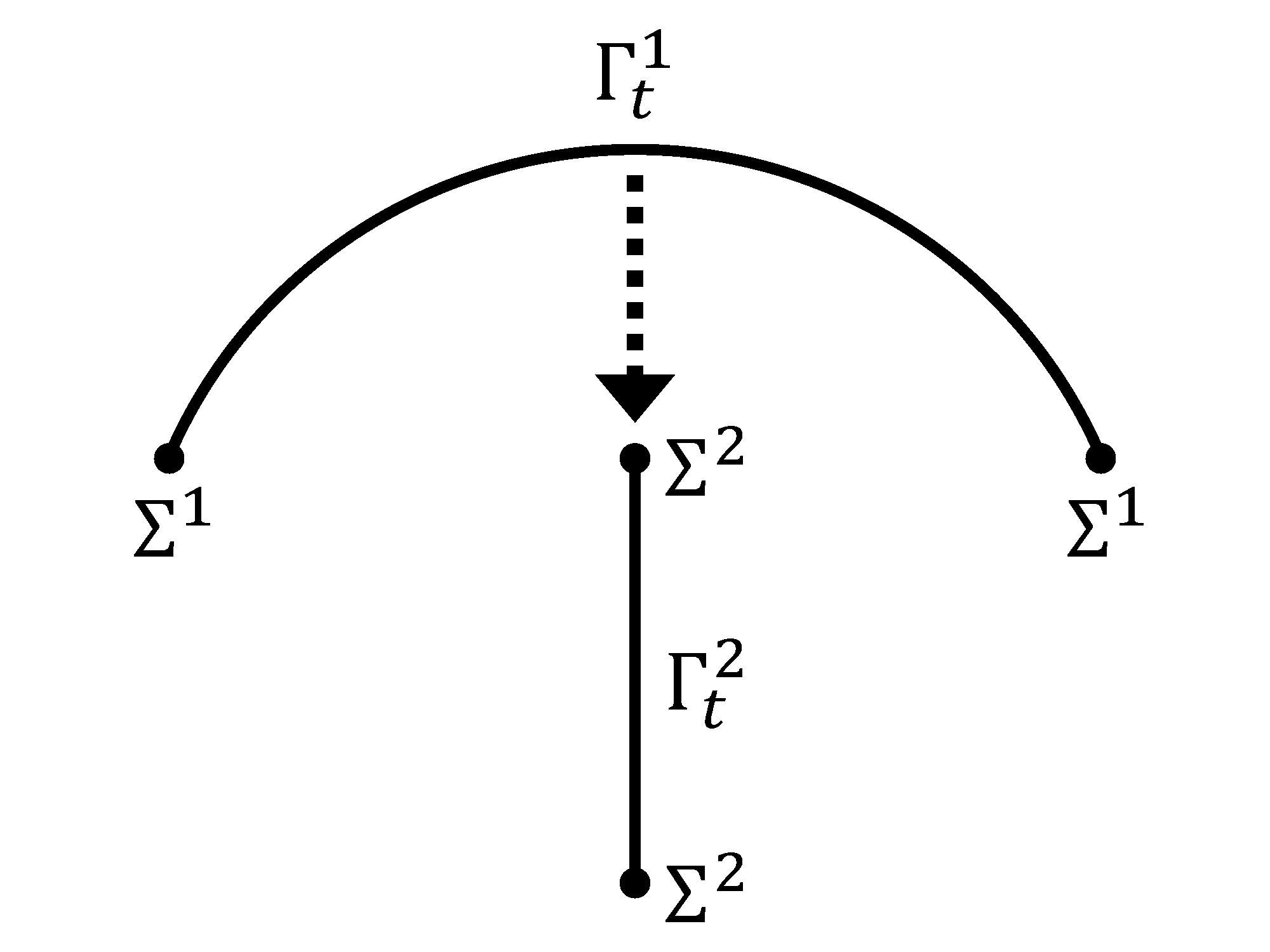}
\caption{$\Gamma_t^1$ hits the upper part of $\Sigma^2$ in finite time.} \label{FColl}
	\end{minipage}
\end{figure}
 The assumption $\inf\left\{\operatorname{dist}(\Gamma_t^2,\Sigma^1)\bigm|t\in(0,T)\right\}>0$ is also necessary as Figure \ref{FColl2} shows.
\begin{figure}[htb]
\centering
\includegraphics[width=5cm]{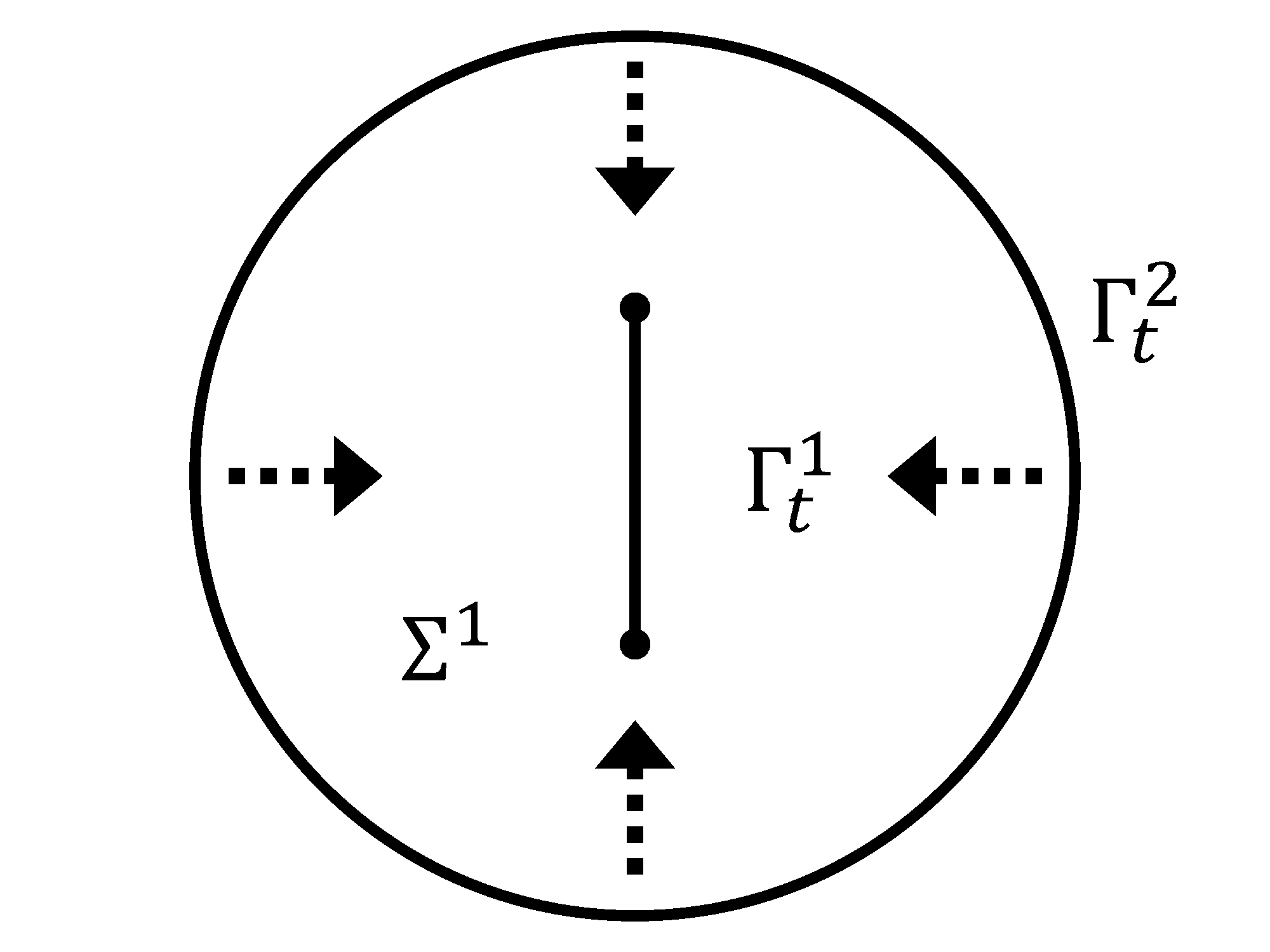}
\caption{$\Gamma_t^2$ hits $\Sigma^1$} \label{FColl2}
\end{figure}
\begin{remark} \label{Rfat}
As already indicated, there are instant fattening phenomena.
 Such an example is given in \cite{M}.
 We recall his example and see the behavior right after the fattening.
 Let $\Gamma_0$ be a unit circle (centered at the origin) and $\Sigma$ consists of three different points on $\Gamma_0$.
 (In \cite{M} $\Sigma$ forms an equilateral triangle, but it is unnecessary.)
 Then $\Gamma_t$ has an interior instantaneously and its outer boundary consists of the three curvature flow with the Dirichlet boundary condition.
 Its inner boundary is the circle of radius $\sqrt{1-2t}$ centered at the origin.
 See Figure \ref{Ffat}.
\begin{figure}[htb]
\centering
\includegraphics[width=9cm]{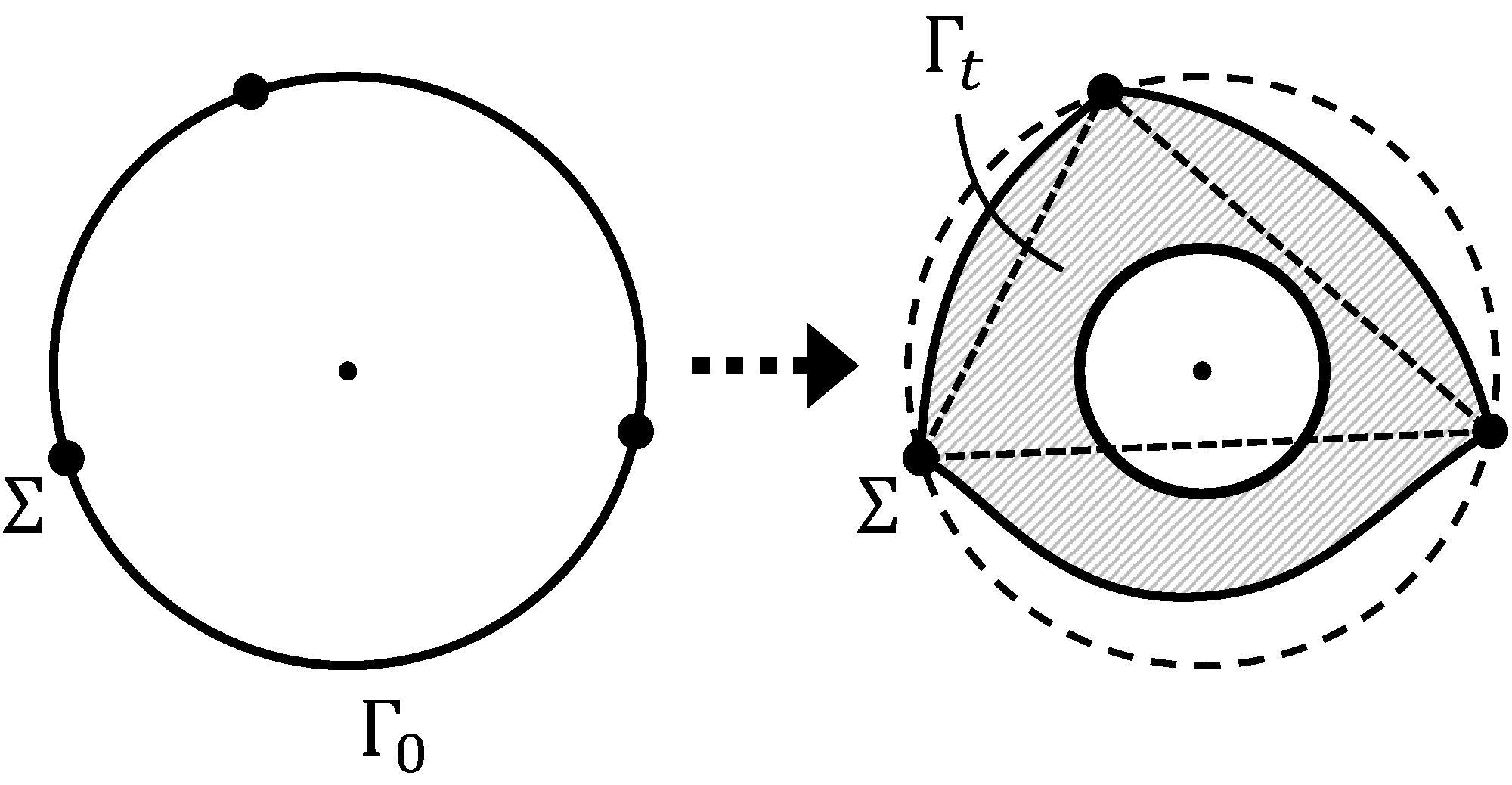}
\caption{fattening} \label{Ffat}
\end{figure}
Note that as in the case without obstacle \cite{G}, it is possible to prove that $t\longmapsto\Gamma_t$ is left continuous in the Hausdorff distance sense, and this fattening example of course does not violate this continuity despite $\Gamma_t$ has an interior. 
\end{remark}
%

\section{Dirichlet problems versus obstacle problems} \label{S3} 

A level-set method for the Dirichlet problem has been established only when the initial surface $\Gamma_0$ is contained in a strictly mean-convex domain \cite{SZ}.
 We shall compare a level-set mean curvature flow with an obstacle to their solution.
In other words, we shall prove Theorem \ref{Tcsz}. 
 We first recall their main theorems.
\begin{prop}[{\cite[Theorems 2.5, 2.6]{SZ}}] \label{PSZ1}
Let $U$ be a bounded $C^2$ domain in $\mathbb{R}^n$.
 Assume that $\partial U$ is strictly mean-convex.
 Let $v_0\in C(\overline{U})$ and $g\in C(\partial U)$ with $v_0=g$ on $\partial U$.
 Then there is a unique viscosity solution $v\in C\left(\overline{U}\times[0,\infty)\right)$ to the level-set equation \eqref{Elev} in $U$ with $v=g$ on $\partial U$ and $\left.v\right|_{t=0}=v_0$.
\end{prop}
\begin{prop}[{\cite[Remark 2.8]{SZ}}] \label{PSZ2}
The set $\Gamma_t^U=\left\{x\in\mathbb{R}^n \bigm| v(x,t)=0 \right\}$ depends only on $\Gamma_0=\left\{x\in\overline{U} \bigm| v_0(x)=0 \right\}$.
\end{prop}
We shall compare our level-set flow with obstacle to this flow.
 For the proof of Theorem \ref{Tcsz}, we recall following general principles which are well known for the level-set equations without obstacles; see e.g.\ \cite{G}.
 However, it is not difficult to extend such results to obstacle problems.
\begin{lemma}[Invariance lemma] \label{Linv}
Assume that (F1), (F2) and \eqref{Ege}.
 Let $D$ be a domain in $\mathbb{R}^n$ and $T>0$.
 Let $\theta$ be a continuous, nondecreasing function in $\mathbb{R}$.
 If $u$ is a viscosity subsolution {\rm(}resp., \ supersolution{\rm)} of the level-set equation \eqref{EF1} of \eqref{EGEN} in $D\times(0,T)$ with obstacles $\psi^\pm$, then $(\theta\circ u)(x,t)=\theta\left(u(x,t)\right)$ is a viscosity subsolution {\rm(}resp., supersolution{\rm)} with obstacles $\theta\circ\psi^\pm$.
 If the equation \eqref{EGEN} is orientation-free, the monotonicity of $\theta$ is unnecessary.
 In particular, $\theta(\sigma)=|\sigma|$ is allowed.
\end{lemma}
See e.g.\ \cite[Theorem 4.2.1]{G}, \cite{CGG}, \cite{ES} without $\psi^+$ for the proof. 
\begin{lemma}[Comparison principle] \label{Lcompa}
Assume that (F1) -- (F3).
 Let $D$ be a bounded domain in $\mathbb{R}^n$ and $T>0$.
 Let $u$ and $v$, respectively, be a viscosity sub- and supersolution of \eqref{EF1} in $D$ with upper obstacle $\psi^+\in C\left(\overline{D}\times[0,T)\right)$.
 Assume that $u^*\leq v_*$ on $(\partial D\times[0,T))\cup (D\times\{0\})$.
 Then $u^*\leq v_*$ in $\overline{D}\times[0,T)$.
\end{lemma}
See e.g. \cite[Theorem 3.3.1]{G} without $\psi^+$ for the proof.

We modify a mean-convex domain $U$.
 For a set $S\subset \mathbb{R}^n$ and $\delta>0$, let $U_\delta(S)$ denote its $\delta$-neighborhood, i.e.,
\[
	U_\delta(S) = \left\{ x\in\mathbb{R}^n \bigm|
	\operatorname{dist}(x,S) < \delta \right\}.
\]
If $S$ is a singleton, i.e., $S=\{p\}$, $p\in\mathbb{R}^n$, $U_\delta(S)$ is nothing but an open ball $\mathring{B}_\delta(p)$ centered at $p$ with radius $\delta$.
\begin{lemma} \label{Lmod}
Let $U$ be a bounded $C^2$ domain in $\mathbb{R}^n$ with strictly mean-convex boundary $\partial U$.
 Assume that $\Sigma$ is a compact set in $\partial U$.
 Then, for $\delta,\delta'>0$ with $\delta'<\delta$, there is a bounded $C^2$ domain $V\subset U$ in $\mathbb{R}^n$ with strictly mean-convex boundary $\partial V$ such that $\overline{V}\backslash U_\delta(\Sigma)\subset U$, $V\cap U_{\delta/2}(\Sigma)=U\cap U_{\delta/2}(\Sigma)$, $V\backslash U_{\delta'}(\partial U)=U\backslash U_{\delta'}(\partial U)$.
\end{lemma}
\begin{proof}
Let $\delta>\delta'>0$ be sufficiently small.  
For a point $p\in\partial U\backslash U_\delta(\Sigma)$, we shall push $\partial U$ a little bit inside near $p$.
 For a neighborhood of $p\in\partial U$, up to a rotation $U$ can be represented as
\begin{equation} \label{Elc}
	U \cap B_R^n(p)
	= \left\{ x = (x',x_n) \in B_R^n(p) \subset \mathbb{R}^n \bigm|
	x_n > h(x'),\ x' \in B_R^{n-1}(p') \right\}
\end{equation}
with some $C^2$ function $h$ satisfying $h(p')=p_n$ provided that $R>0$ is taken sufficiently small, where $B_R^m(q)$ denotes the closed ball in $\mathbb{R}^m$ with radius $R$ centered at $q\in\mathbb{R}^m$ and $p=(p',p_n)$.
 Since $\partial U$ is strictly mean-convex, there is a nonnegative $C^2$ function $\sigma$ which is zero outside $B_{R/2}^n(p)$ such that that $\sigma(x')>0$ in $\mathring{B}_{R/2}^n(p)$ and that
\[
	V(p) = \left\{ x \in B_R^n(p) \bigm|
	x_n > h(x') + \sigma(x'),\ x' \in B_R^{n-1}(p') \right\}
\]
has still strictly mean-convex boundary $\partial V_p$ in $B_R^n(p)$.
 We take $R$ small so that $R<\delta$.
 We then define a domain pushed at $p$ as
\[
	P_p(U) = \left( U \backslash B_R^n(p) \right) \cup V(p).
\]
We may assume that $P_p(U)\backslash U_{\delta'}(\partial U)=U\backslash U_{\delta'}(\partial U)$ by taking $\sigma$ small.
 By definition, $P_p(U)\subset U$ with
\[
	\partial\left(P_p(U)\right)\cap\partial U
	=\partial U \backslash B_{R/2}^n(p) \supset \Sigma
\]
and $\partial\left(P_p(U)\right)$ is still strictly mean-convex, note that $\partial\left(P_p(U)\right)=\partial U$ in $B_R^n(p)\backslash B_{R/2}^n(p)$ so $P_p(U)\cap U_{\delta/2}(\Sigma)=U\cap U_{\delta/2}(\Sigma)$.
 See Figure \ref{Fpu}.
\begin{figure}[htb]
\centering
\includegraphics[width=4cm]{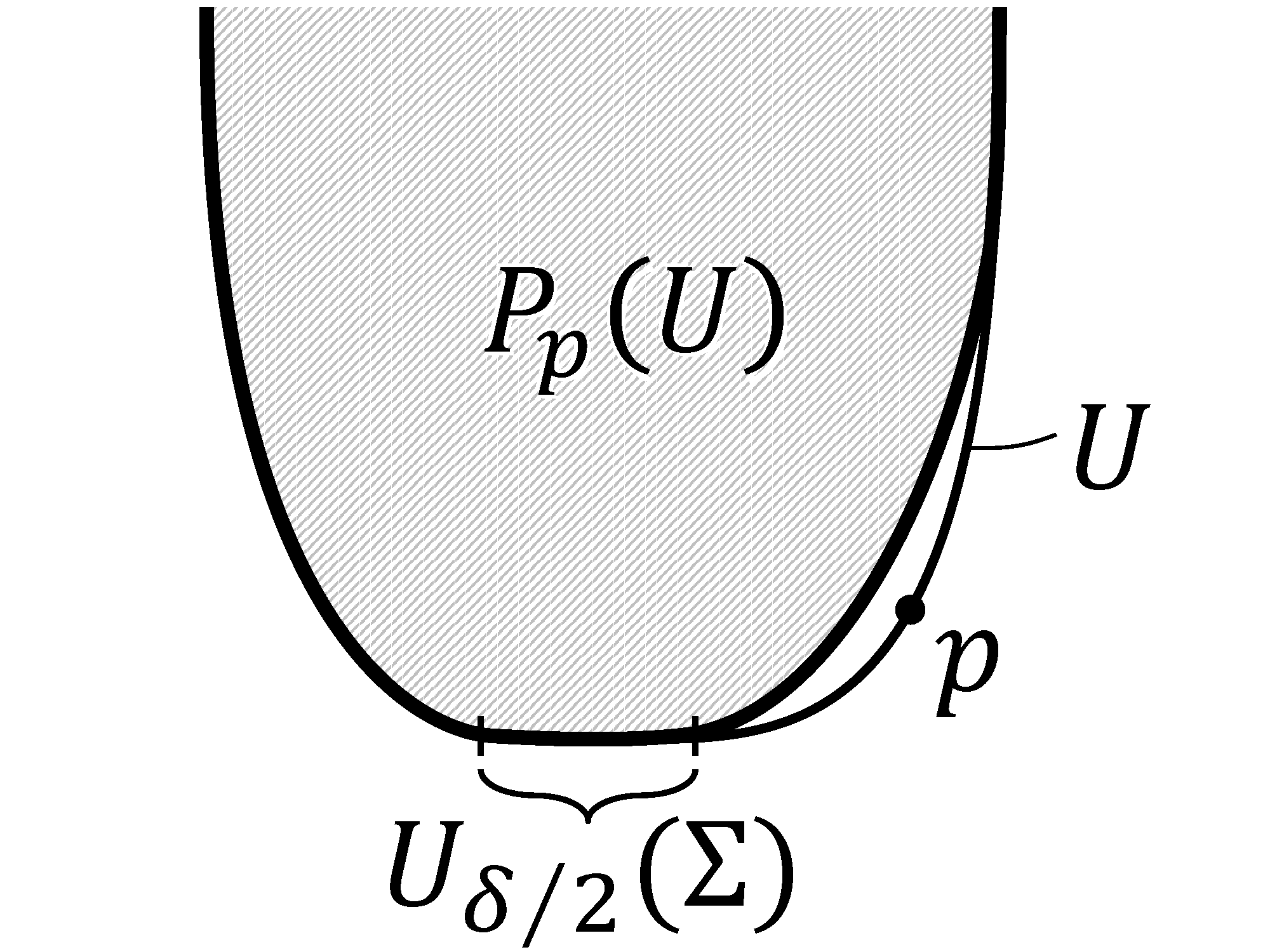}
\caption{a pushed domain} \label{Fpu}
\end{figure}

Since $\partial U\backslash U_\delta(\Sigma)$ is compact, we are able to cover $\partial U\backslash U_\delta(\Sigma)$ by finitely many local coordinate patches.
 More precisely, there are finitely many points $\{p_i\}_{i=1}^k$ and $R>0$ such that $\partial U\backslash U_\delta(\Sigma)\subset\bigcup_{i=1}^k\mathring{B}_{R/2}^n(p_i)$ and in each $B_R^n(p_i)$ the set $U\cap B_R^n(p)$ is represented as \eqref{Elc} up to translation.
 We may assume that $|p_i-p_j|>R/2$ if $i\neq j$, $1\leq i,j\leq k$.
 We push $U$ in finitely many times to get desired $V$.
 More precisely, we set
\[
	V_1 = U, \quad
	V_j = P_{p_{j-1}}(V_{j-1})\ \text{for}\ j=2,\ldots,k+1, \quad
	V = V_{k+1}. 
\]
Since $|p_i-p_j|>R/2$ for $i\neq j$, we see that $p_j\in\partial V_{j-1}$ and $B_{p_j}^n(R)$ is still coordinate patch of $V_{j-1}$ provided that $\sigma$ is taken sufficiently small for each step.
 Thus our construction is well defined.
 By the property of pushing, we easily observe that our $V$ satisfies all desired properties.
\end{proof}
For a strictly mean-convex domain $U$, the distance function $\operatorname{dist}(x,U)$ is a standing (time-independent) viscosity subsolution of the level-set equation in some neighborhood of $\overline{U}$.
\begin{lemma} \label{Lmcst}
Let $U$ be a bounded $C^2$ domain in $\mathbb{R}^n$ with strictly mean-convex boundary.
 Then $\operatorname{dist}(x,U)$ is a standing viscosity subsolution of the level-set equation \eqref{Elev} in $U_\delta(\overline{U})$ for sufficiently small $\delta>0$.
\end{lemma}
\begin{proof}
Let $\mathbf{n}$ be the inward unit normal vector field of $\partial U$.
Let $\kappa_i(x)$ ($1\leq i \leq n-1$) be principal curvatures of $\partial U$ in the direction of $\mathbf{n}$.
 As in \cite[Chapter 14, Appendix]{GT} (see also the proof of Lemma \ref{Lsigma}), the principal curvatures $\kappa_i^d(y)$ 
for $d\in\R$ of
\[
	S_d = \left\{ y = x - d\mathbf{n}(x) \bigm| x \in \partial U \right\}
\]
equals
\[
	\kappa_i^d(y) = \frac{\kappa_i(x)}{1+d\kappa_i(x)}.
\]
The strict mean convexity implies that $\inf_{\partial U}H>0$, where $H(x)=\sum_{i=1}^{n-1}\kappa_i(x)$.
 Thus, for sufficiently small $|d|$, say $|d|<\delta$,
\[
	H_d = \sum_{i=1}^{n-1}\kappa_i^d(y) > 0
	\quad\text{on}\quad y \in S_d.
\]
We set a signed distance
\begin{equation*}
	w(x)= \left \{
\begin{array}{l}
	\operatorname{dist}(x,U), \quad
	x \in \mathbb{R}^n \backslash U \vspace{0.2em} \\
	-\operatorname{dist}(x,\mathbb{R}^n \backslash U), \quad
	x \in U
\end{array}
\right.
\end{equation*}
and observe that $w$ is $C^2$ in a neighborhood of $\partial U$ since $\partial U$ is $C^2$ (see \cite{GT, KP}).
Since $\mathbf{n}=-\nabla w/|\nabla w|$ and $H_d=-\operatorname{div}\mathbf{n}$, we conclude that
\[
	-\operatorname{div} \left( \nabla w/|\nabla w| \right) \leq 0
\]
in $U_\delta(\partial U)$.
 Thus, $w$ is a standing (time-independent) viscosity subsolution of \eqref{Elev} in $U_\delta(\partial U)$.
 By definition, $w\vee0=\max(w,0)$ is also a viscosity subsolution.
 Thus, $\operatorname{dist}(x,U)$ is a viscosity subsolution of \eqref{Elev} in $U_\delta(U)$.
\end{proof}

To compare two level-set flows, it is convenient to recall a renormalization lemma.
\begin{lemma} \label{Lnor}
Let $g_1$ and $g_2$ are two nonnegative continuous functions on a compact set $K$ in $\mathbb{R}^n$.
 Assume that
\[
	\left\{ x \in K \bigm| g_2(x) = 0 \right\}
	\subset \left\{ x \in K \bigm| g_1(x) = 0 \right\}.
\]
Then, there is $\theta\in C([0,\infty))$ which is increasing and $\theta(0)=0$ such that
\[
	g_1 \leq \theta \circ g_2 \quad\text{on}\quad K.
\]
\end{lemma}
For the proof, see \cite[Lemma 4.2.9]{G}.
 Here is a basic idea.
 We set
\[
	\widetilde{\theta}(\sigma) = \sup\left\{ g_1(z) \bigm| z \in K,\ g_2(z) \leq \sigma \right\}
\]
and observe that $\widetilde{\theta}(\sigma)\downarrow0$ as $\sigma\downarrow0$ and nondecreasing.
 It is not difficult to construct our desired $\theta\geq\widetilde{\theta}$.

We are now in position to prove Theorem \ref{Tcsz}.
\begin{proof}[Proof of Theorem {\rm\ref{Tcsz}}]
Let $u$ be the viscosity solution of the level-set equation \eqref{Elev} with upper obstacle $\psi^+(x)=\operatorname{dist}(x,\Sigma)$ and initial data $u_0=\operatorname{dist}(x,\Gamma_0)$.
 By Lemma \ref{Lmod}, for $\delta>0$ there is a bounded $C^2$ domain $V\subset U$ with strictly mean-convex boundary $\partial V$ such that $\overline{V}\backslash U_\delta(\Sigma)\subset U$, $V\cap U_{\delta/2}(\Sigma)=U\cap U_{\delta/2}(\Sigma)$, $V\backslash U_\delta(\partial U)=U\backslash U_\delta(\partial U)$.
 By Lemma \ref{Lmcst}, $\operatorname{dist}(x,V)$ is a standing viscosity subsolution of the level-set equation in $U_\delta(\overline{V})$.
 By the invariance Lemma \ref{Linv}, $w=\delta/2\wedge \operatorname{dist}(x,V)$ is a viscosity subsolution of the level-set equation in $\mathbb{R}^n\times[0,\infty)$.
 Since $w\leq\psi^+$, $w$ is also a viscosity subsolution with obstacle $\psi^+$.
 By Lemma \ref{Linv}, $u^\delta=u\wedge\delta$ is a viscosity solution with obstacle $\psi^+$ in $\mathbb{R}^n\times(0,\infty)$.
 We may assume $\Gamma_0\subset V$ by taking $\delta$ and $\delta'$ small since we assume $\Gamma_0\setminus\Sigma\subset U$.
 Thus $w\leq u_0\wedge\delta$ at $t=0$.
 By the comparison principle (Lemma \ref{Lcompa}) in $B_R(0)\times(0,T)$ for a large $R$ such that $u^\delta=\delta$ and $w=\delta/2$ outside $B_R$, we observe that $w\leq u$.
 In particular $u>0$ outside $\overline{V}$ so $\Gamma_t\backslash\overline{V}=\emptyset$ for all $t\geq0$.
 In particular, $\Gamma_t\backslash U_{\delta/2}(\Sigma)\subset U$ for all $t\geq0$.
 Since $\delta>0$ is arbitrary, this implies that $\Gamma_t\backslash\Sigma\subset U$ so that $u>0$ on $\partial U\backslash\Sigma$.

We shall prove that $\Gamma_t^U\subset\Gamma_t$ in $U$.
 For the viscosity solution $v$ of the Dirichlet problem \eqref{eq:CD}, we take
\[
	\psi^+(x,t) = v(x,t) + \sigma(x) \quad\text{in}\quad \overline{U}\times[0,\infty),
\]
where $\sigma\in C(\overline{U})$ is positive in $U$ and $\sigma=0$ on $\partial U$.
 For $T>0$ we extend $\psi^+$ outside $\overline{U}\times[0,T]$ continuously such that $\psi^+>0$ on $\overline{U}^c\times[0,T]$ such that it equals a positive constant $c$ on $B_R(0)^c\times[0,T]$ for a large $R>0$ such that $B_R(0)\supset\overline{U}$.  
 Here, $A^c$ denotes the complement of $A$, i.e., $A^c:=\R^n\setminus A$. 
Let $\tilde{\psi}^+$ be such an extension.
 We take $\Psi^+=\tilde{\psi}^+\vee\left(\operatorname{dist}(x,\overline{U})\wedge c\right)$ so that $\Psi^+>0$ on $\overline{U}^c\times[0,T]$.
 Note that $\Psi^+(x,t)=0$ if and only if $(x,t)\in\Sigma\times[0,\infty)$ so $\Psi^+$ is an obstacle function of $\Sigma$.
 We extend $v$ outside $U$ so that $v=\Psi^+$ in $U^c\times[0,T)$.
 Let $u$ be the viscosity solution of the level-set equation with upper obstacle $\Psi^+$ and initial data $u_0=v_0=\left.v\right|_{t=0}$ in $\mathbb{R}^n$.
 By definition, $v$ is a viscosity supersolution of the level-set equation in $B_R(0)\times(0,T)$ with upper obstacle $\Psi^+$.
 By comparison principle for the obstacle problem in $B_R(0)$ (Lemma \ref{Lcompa}), we conclude that $u\leq v$ in $B_R(0)\times(0,T)$.
 Since $v=\Psi^+>0$ outside $\overline{U}$, this implies that $\Gamma_t^U\subset\Gamma_t$.

It remains to prove $\Gamma_t\subset\Gamma_t^U$ in $U$.
 As we already observed, $u\leq v$ in $U\times(0,T)$ so $u$ is a viscosity solution of the level-set equation in $U\times(0,T)$ (with no obstacle).
 The property $\Gamma_t\backslash\Sigma\subset U$ implies that $u>0$ on $\partial U\backslash\Sigma$.
 We set
\[
	K = (\partial U \times [0,T]) \cup (\overline{U} \times \{0\})
\]
and $g_1=\left.v\right|_K$, $g_2=\left.u\right|_K$ and observe that
\[
	\left\{ (x,t) \in K \bigm| g_2(x,t) = 0 \right\}
	= \left\{ (x,t) \in K \bigm| g_1(x,t) = 0 \right\}	
	= \Sigma \times [0,T] \cup \Gamma_0 \times \{0\}.
\]
By the renormalizing Lemma \ref{Lnor}, there is an increasing function $\theta\in C([0,\infty))$ with $\theta(0)=0$ and $g_1\leq\theta\circ g_2$ on $K$.
 By the invariance Lemma \ref{Linv}, $\theta\circ u$ is a viscosity solution of the level-set equation in $U\times(0,T)$ with initial data $\theta\circ u_0$ but without any obstacle.
 By the comparison principle (without obstacle) in $U\times(0,T)$, we conclude that $v\leq\theta\circ u$.
 This implies $\Gamma_t\subset\Gamma_t^U$ in $U$.
 The proof is now complete.
\end{proof}

\section{Consistency with smooth solutions} \label{S4} 

We shall prove Theorem \ref{Tcons}.
 For this purpose, we shall construct suitable viscosity sub- and supersolutions based on $\Gamma_t^s$ as in \cite{ES}, 
 where $\{\Gamma_t^s\}_{0\le t\le T}\subset\mathbb{R}^n$ is a family of compact hypersufraces given in Theorem \ref{Tcons}. 
 We first observe that the distance function $\operatorname{dist}(x,\Sigma)$ of $\Sigma$ is a standing viscosity subsolution of the level-set equation \eqref{Elev} near $\Sigma$.
\begin{lemma} \label{Lsigma}
Let $\Sigma$ be a $C^2$ compact $k$ codimensional {\rm(}$n-k$ dimensional{\rm)} manifold embedded in $\mathbb{R}^n$ possibly with boundary. 
\begin{enumerate}
\item[{\rm(i)}] There exists $\delta_0>0$ such that the set $U_\delta(\Sigma)$ is a strictly mean-convex bounded $C^2$ domain in $\mathbb{R}^n$ for $\delta\in(0,\delta_0)$ provided that $\Sigma$ has no geometric boundary. 
\item[{\rm(i\hspace{-0.15em}i)}] The function $\operatorname{dist}(x,\Sigma)$ is a standing viscosity subsolution of the level-set equation in $U_\delta(\Sigma)$ including the case when $\Sigma$ has $n-k-1$ dimensional $C^2$ geometric boundary. 
\end{enumerate}
\end{lemma}
\begin{proof}
The proof is more or less known \cite[Remark 4]{Amb}, but we give it for completeness. 

We first prove (i). 
As in \cite{GT, KP}, the distance function $d(x):=\operatorname{dist}(x,\Sigma)$ is $C^2$ in $U_\delta(\Sigma)\backslash\Sigma$ for a small $\delta$, say $\delta\in(0,\delta_1]$.
 Indeed, $C^2$ regularity implies that  there is the unique $y(x)\in\Sigma$ such that $\left|x-y(x)\right|=d(x)$ for $x\in U_\delta(\Sigma)$ for small $\delta>0$.
 Note that $y(x)$ is a critical point of $C^2$ function $|x-y|$ as a function of $y\in\Sigma$ for $x\in U_\delta(\Sigma)\backslash\Sigma$ and its differential in $y$ is of full rank.
 By the implicit function theorem $y(x)$ is $C^1$ in $x\in U_\delta(\Sigma)\backslash\Sigma$ for small $\delta>0$.
 Since
\[
	x = y(x) + d(x) \nabla d(x)
\]
and $y$ is $C^1$, $\nabla d(x)$ is $C^1$ so that $d\in C^2\left(U_\delta(\Sigma)\backslash\Sigma\right)$.

We may assume that $U_\delta(\Sigma)$ is a $C^2$ domain for $\delta\in(0,\delta_1]$.
 As discussed in the proof of Lemma \ref{Lmcst}, we are interested in the evolution of principal curvatures of $S_\delta=\partial U_\delta(\Sigma)$ as $\delta\downarrow0$.
 This is nowadays standard.
 See e.g.\ \cite[Theorem 3.2]{AS}, \cite[Theorem 2.2]{AM} where evolution of $\nabla^2d^2/2$ is studied.
 We consider the evolution of the Hessian matrix $\nabla^2d=(d_{ij})_{1\leq i,j\leq n}$ of $d$, where $d_i=\partial d/\partial x_i$, $d_{ij}=\partial^2d/\partial x_i\partial x_j$.
 Assuming for the moment that $\Sigma$ is $C^3$ so that $d_{ij}$ is $C^1$,
 we set
\[
	M(t) = d_{ij} (y+tp), \quad y \in \Sigma, \quad t>0
\]
for a unit vector $p$ orthogonal to the tangent space $T_y\Sigma$ of $\Sigma$ at $y$.
 We differentiate in time to get
\[
	\frac{d}{dt} M(t) = \sum_{k=1}^n d_{ijk} (y+tp)p_k
	= \sum_{k=1}^n d_{ijk} d_k (y+tp)
\]
since $p=\nabla d(y+tp)$.
 Since $\left(\sum_{k=1}^n d_k^2\right)_i=0$ so that $\sum_{k=1}^n d_{ik}d_k=0$, we differentiate in $x_j$ to get
\[
	\sum_{k=1}^n d_{ijk} d_k
	+ \sum_{k=1}^n d_{ik} d_{kj} = 0.
\]
Thus
\begin{equation} \label{EESec}
	\frac{d}{dt} M(t) = -M(t)^2, \quad
	0 < t < \delta.
\end{equation}
In the case that $\Sigma$ is just $C^2$, we interpret $d_{ijk}$ as a distribution.
 For a derivative $\partial f/\partial x_i$ of a continuous function $f$ in a domain $D\subset \mathbb{R}^d$, one may multiple $g\in C^1(D)$.
 Indeed, we define
\[
	\left\langle g \frac{\partial f}{\partial x_i},\varphi \right\rangle
	:= -\int_D f \frac{\partial(g\varphi)}{\partial x_i} dx
\]
for compactly supported smooth function $\varphi$ in $D$.
 Thus \eqref{EESec} is also obtained in the sense of distribution, which agrees with conventional solution.

By the evolution equation \eqref{EESec}, the eigenvector is chosen independent of $t$.
 Let $P$ be the orthogonal projection to the tangent space $T_y\Sigma$ from $\mathbb{R}^n$.
 Then $\lim_{t\downarrow0}\left(-PM(t)P\right)$ is the second fundamental form of $\Sigma$ at $y$ in the direction of $p$.
 Thus, its eigenvalues are principal curvatures of $\Sigma$ in the direction of $p$ (cf.\ \cite[Theorem 3.5]{AM}).
 Since $PM^2P=(PMP)^2$, the evolution law \eqref{EESec} implies that $n-k$ of the (outward) principal curvature $\kappa_i^\delta(x)$ at $x\in S_\delta$ equals
\[
	\kappa_i^\delta(x) = \frac{\kappa_i(y,p)}{1-\delta\kappa_i(y,p)}, \quad
	x = y+\delta p, \quad 1 \leq i \leq n-k,
\]
where $\kappa_i(y,p)$ denotes the $i$ the principal curvature of $\Sigma$ in the direction of $p$.
 Here we invoke the fact that the eigenvalue of $-PM(\delta)$ is the (outward) principal curvature of $S_\delta$ so that \eqref{EESec} is reduced to
\[
	\frac{d}{dt} \kappa_i^t = \left(\kappa_i^t\right)^2, \quad
	1 \leq i \leq n-k,
\]
which yields the desired formula of $\kappa_i^\delta$.
 By \eqref{EESec}, $M(t)t$ is bounded in $t\in(0,1)$.
 Thus $\eta=d^2/2$ is $C^2$ since
\[
	\nabla^2 \eta = d \nabla^2 d + \nabla d \otimes \nabla d.
\]
Since $|\nabla d|=1$, this implies that
\[
	(\nabla^2 \eta)p = p \quad\text{with}\quad 
	p = \nabla d.
\]
In other words, $\nabla^2d\nabla d=\nabla d/d$.
 Thus the remaining principal curvature $\kappa_{j}^\delta(x)$ ($n-k+1\leq j\leq n-1$) at $x\in S_\delta$ equals $-1/\delta$.
 In other words, one of inward curvature of $S_\delta$ equals $1/\delta$.
 We now complete that the inward mean curvature $H$ has the form
\begin{align*}
	H &= \frac{k-1}{\delta} - \sum_{i=1}^{n-k} \kappa_i^\delta(x) \\
	&\geq \frac{k-1}{\delta}
	-\sup \left\{ \sum_{i=1}^{n-k}\frac{\kappa_i(y,p)}{1-\delta\kappa_i(y,p)} \biggm| y\in\Sigma,\ p\in T_y\Sigma,\ |p| = 1 \right\}.
\end{align*}
Since the second term in the rightest-hand side is bounded in $\delta$, we conclude that $\inf_{x\in S_\delta}H(x)>0$ for sufficiently small $\delta$.
 See Figure \ref{Fthin}.
\begin{figure}[htb]
\centering
\includegraphics[width=5.2cm]{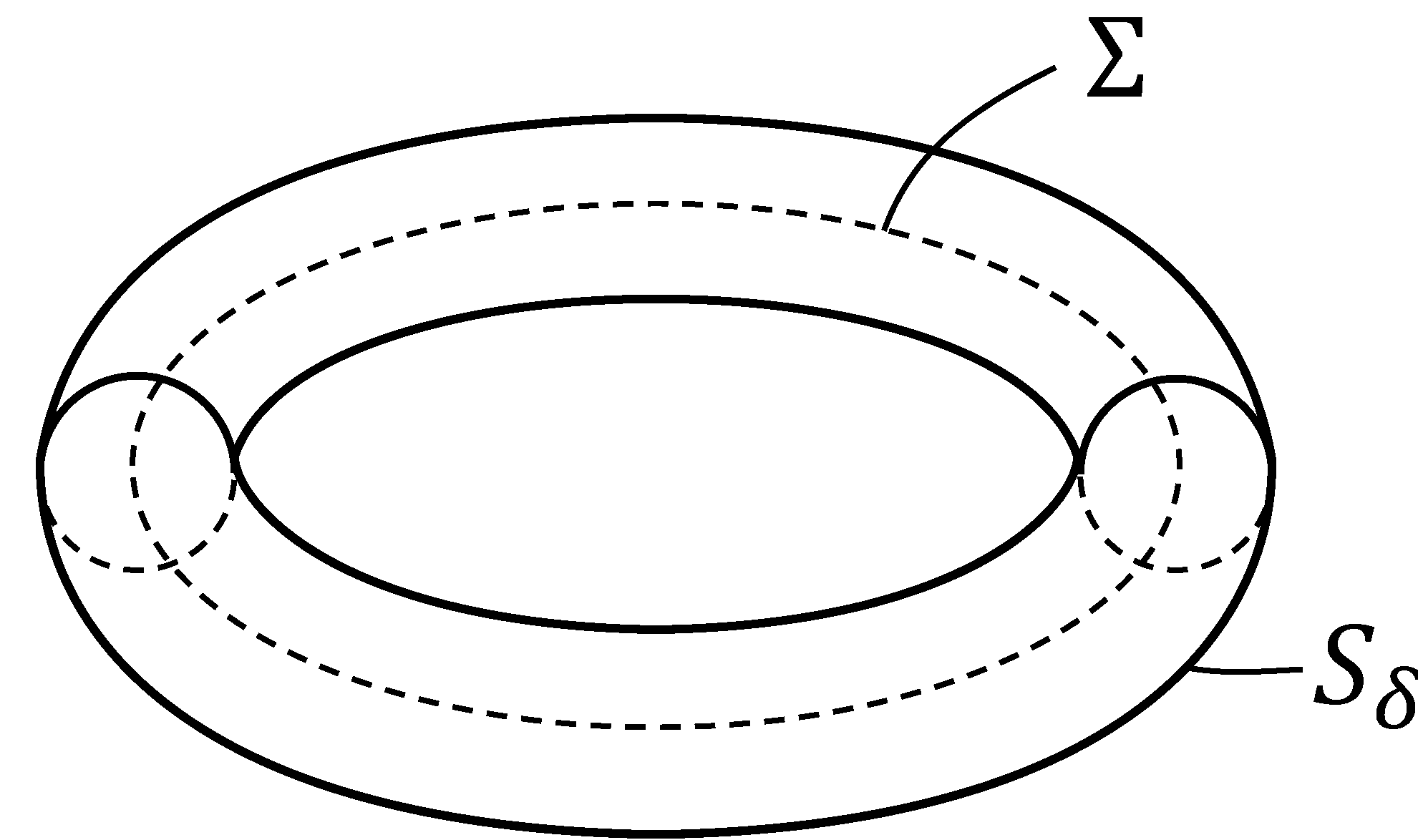}
\caption{$\Sigma$ and $S_\delta$} \label{Fthin} 
\end{figure}

We next prove (ii). 
We first consider the case when $\Sigma$ has no geometric boundary. 
 By direct calculation, for $v(x)=\operatorname{dist}(x,\Sigma)$, we have
\[
	v_t - |\nabla v| \operatorname{div} \left(\frac{\nabla v}{|\nabla v|}\right)
	= - \operatorname{div} \left(\frac{\nabla v}{|\nabla v|}\right)
	= - H_\delta(y),\ y \in S_\delta.
\]
By (i), $H_\delta\geq0$ for $\delta\in(0,\delta_0)$.
Thus $v$ is a viscosity subsolution of the level-set equation in $U_{\delta_0}(\Sigma)$.

If $\Sigma$ has a geometric boundary $b\Sigma$, $b\Sigma$ must have no boundary.
 Thus for small $\delta>0$ $U_\delta(b\Sigma)$ is a bounded $C^2$ mean-convex domain by (i).
 A simple modification of the proof of (i) yields that $S_\delta$ is $C^2$ and its boundary has a positive inward mean curvature for small $\delta>0$ outside the set
\[
	Z = \left\{ x \in \mathbb{R}^n \bigm|
	\operatorname{dist}(x,b\Sigma) \le \operatorname{dist}(x,\Sigma) \right\}.
\] 
Thus $S_\delta$ is $C^2$ outside the boundary of $Z$ and its inward mean curvature is positive.
 It is not difficult to see that $Z$ is $C^2$ and $S_\delta$ is $C^1$ across $\partial Z$ (cf.\ Figure \ref{Fbdry}).
\begin{figure}[htb]
\centering
\includegraphics[width=5.5cm]{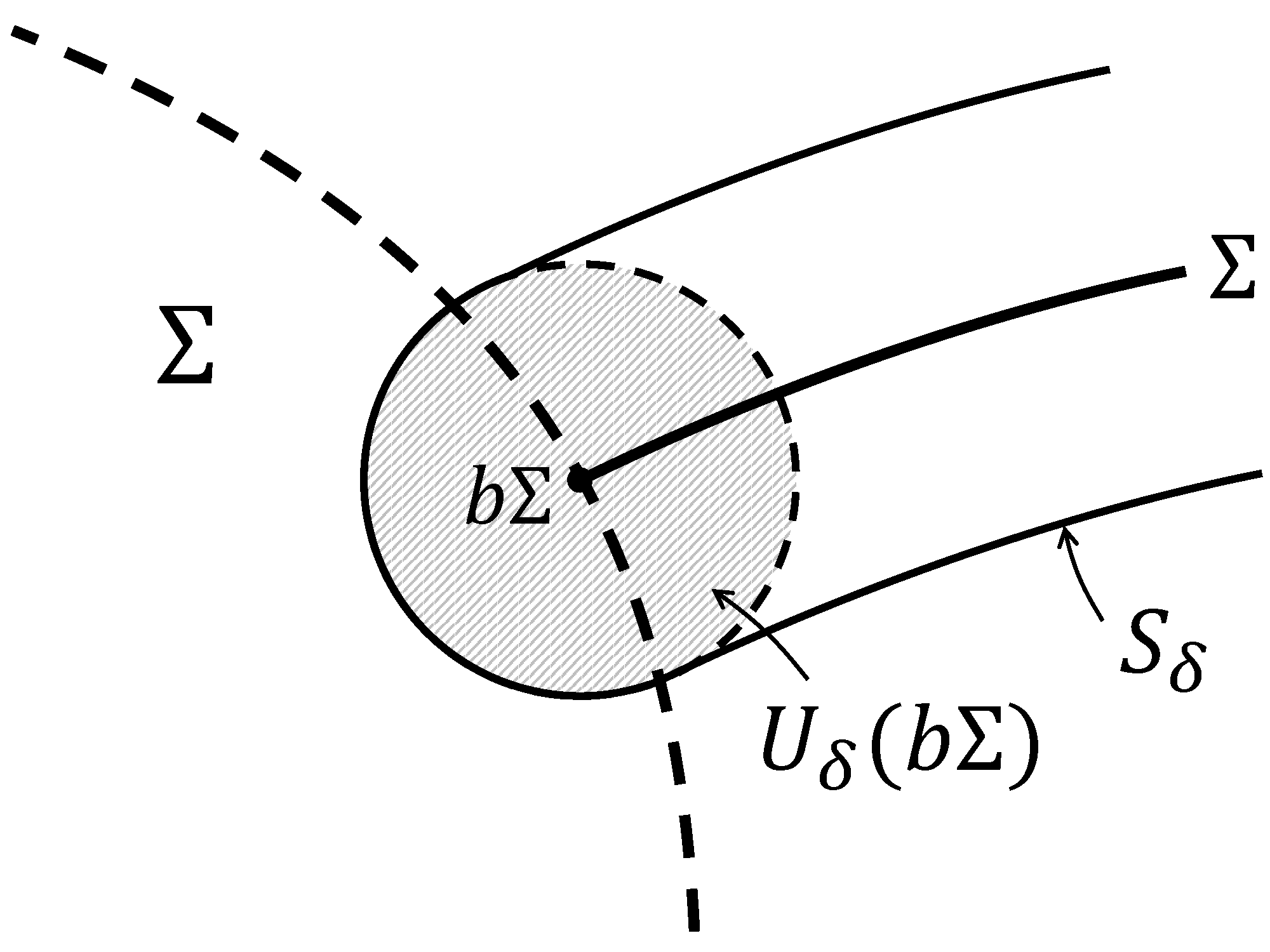}
\caption{near the boundary of $\Sigma$} \label{Fbdry} 
\end{figure}
Thus, $v$ is $C^1$ in $U_{\delta_0}(\Sigma)$ and $C^2$ in $U_{\delta_0}(\Sigma)\backslash\partial Z$.
 Moreover $-\operatorname{div}\left(\nabla v/|\nabla v|\right)\leq0$ outside $Z$.
 Since the second derivatives of $v$ are continuously extended on $Z\cap U_{\delta_0}(\Sigma)$ and $U_{\delta_0}(\Sigma)\backslash Z$, $v$ is a viscosity subsolution of the level-set equation in $U_{\delta_0}(\Sigma)$ at least in viscosity sense.
\end{proof}
\begin{remark} \label{RCuvb}
In \cite[Theorem 4]{Amb}, it is proved that the eigenvalues of $\nabla^2\eta(x)$ equals
\[
	\frac{-\delta\kappa_i(y,p)}{1-\delta\kappa_i(y,p)}
\]
in the direction tangential to $\Sigma$ and $1$ in the direction normal to $\Sigma$, where $\kappa_i(y,p)$ is the principal curvature in the direction of $p$.
 Thus one of principal curvatures in the inward direction of $S_\delta$ must be $1/\delta$, so it is rather clear to see it blows up as $\delta\to0$ as remarked in \cite[Remark 4]{Amb}.
 In \cite{Amb, AM, AS}, it is assumed that $\Sigma$ is smooth through its codimension is not necessarily $2$.
 In our proof, we clarify that $C^2$-regularity of $\Sigma$ is sufficient.
\end{remark}

We next construct viscosity suitable sub- and supersolution for the obstacle problem based on $\Gamma_t^s$.
\begin{lemma} \label{Lsubsuper}
Let $\{\Gamma_t^s\}_{0\leq t\leq T}$ be a mean curvature flow in Theorem {\rm\ref{Tcons}} with $\left.\Gamma_t^s\right|_{t=0}=\Gamma_0$ whose boundary is $\Sigma$.
 For $\delta>0$, set
\[
	w(x,t) = \operatorname{dist}(x,\Gamma_t^s) \wedge \delta, \quad
	t \in [0,T].
\]
Then $w$ is a viscosity supersolution of the level-set equation \eqref{Elev} with upper obstacle $\psi^+(x)=\operatorname{dist}(x,\Sigma)$ in $\mathbb{R}^n\times(0,T)$ for sufficiently small $\delta>0$.
 The function $z=e^{-\lambda t}w$ is a viscosity subsolution of the level-set equation \eqref{Elev} in $\mathbb{R}^n\times(0,T)$ for sufficiently small $\delta>0$ and sufficiently large $\lambda>0$.
\end{lemma}
\begin{proof}
We set
\begin{align*}
	E^\delta &= \left\{ (x,t) \in \mathbb{R}^n \times [0,T] \bigm|
	\operatorname{dist}(x,\Sigma) \leq \operatorname{dist}(x,\Gamma_t^s),\ \operatorname{dist}(x,\Sigma) \leq \delta\right\}, \\
	E^\delta(t) &= \left\{ x \in \mathbb{R}^n \bigm|
	(x,t) \in E^\delta \right\} \quad\text{for}\quad t \in [0,T];
\end{align*}
see Figure \ref{Few}.
\begin{figure}[htb]
\centering
\includegraphics[width=6cm]{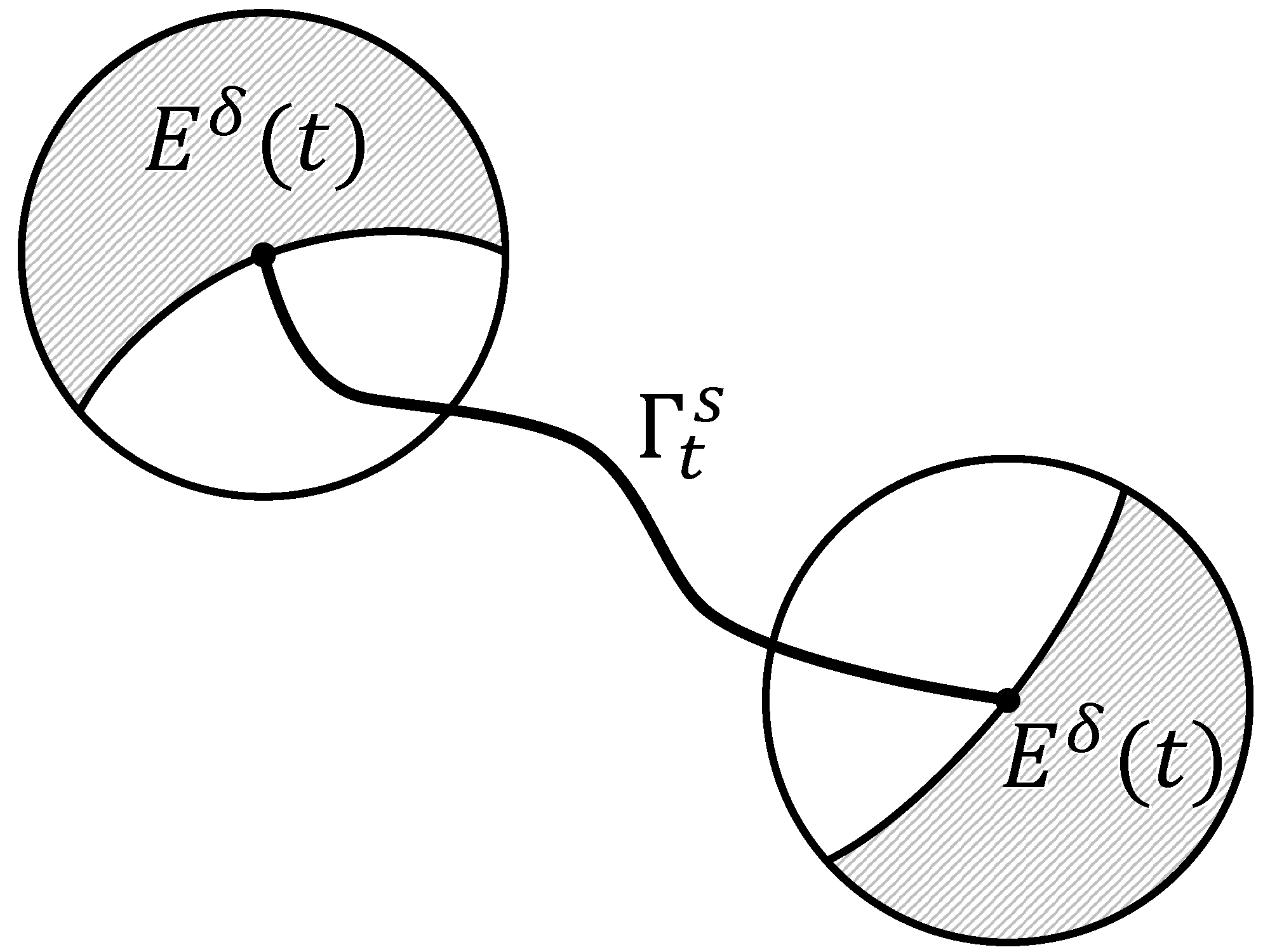}
\caption{sets $E^\delta(t)$ and $\Gamma_t^s$} \label{Few}
\end{figure}
Let $\nu$ be a unit normal vector field of $\Gamma_t^s$ and let $d^s=\operatorname{dist}(x,\Gamma_t^s)$ in the direction of $\nu$ and $d^s=-\operatorname{dist}(x,\Gamma_t^s)$ in the direction of $-\nu$.
 This can be defined in $U_\delta(\Gamma_t^s)\backslash E^\delta(t)$.
 By definition,
\[
	V = -d_t^s \quad\text{on}\quad \Gamma_t^s,
\]
where $V$ is the normal velocity in the direction of $\nu=\nabla d^s$.
 Let $\kappa_i$ be the principal curvatures in the direction of $\nu$ so that $H=\sum_{i=1}^n\kappa_i=-\operatorname{div}\nu$ on $\Gamma_t^s$.
 We consider
\[
	\Gamma_t^{s,\varepsilon} = \left\{ x \in \mathbb{R}^n \bigm|
	x = y + \varepsilon \nu(y),\ y \in \Gamma_t^s \right\}
\]
and observe as in the proof of Lemma \ref{Lsigma} that the principal curvature $\kappa^\varepsilon_i$ of $\Gamma_t^{s,\varepsilon}$ equals
\[
\kappa^\varepsilon_i(x) = \frac{\kappa_i(y)}{1-\varepsilon\kappa_i(y)}
\]
for small $\varepsilon>0$.
 Thus the mean curvature $H^\varepsilon$ of $\Gamma_t^{s,\varepsilon}$ equals
\[
	\sum_{i=1}^{n-1} \frac{\kappa_i(y)}{1-\varepsilon\kappa_i(y)}
\]
which equals $-\operatorname{div}\nabla d^s=-\Delta d^s$ on $\Gamma_t^{s,\varepsilon}$.
 Since the normal velocity of $\Gamma_t^{s,\varepsilon}$ equals $-d_t^s(y,t)$ even on $\Gamma_t^{\delta,\varepsilon}$, we conclude that $d^s$ satisfies
\[
	d_t^s - \Delta d^s = -\sum_{i=1}^{n-1} \left( \kappa_i(y) -  \frac{\kappa_i(y)}{1-d^s\kappa_i(y)} \right)
	= \sum_{i=1}^{n-1} \frac{\kappa_i^2 d^s}{1-d^s\kappa_i(y)} \geq 0
\]
in
\[
	W_+^\delta = \left\{ (x,t) \in E_{2\delta}^c \bigm|
	0 \leq d^s(x,t) < 2\delta,\ 0 < t < T \right\}
\]
for small $\delta>0$.
 Since $|\nabla d^s|=1$, we conclude that
\[
	d_t^s - |\nabla d^s| \operatorname{div}\left(\nabla d^s/|\nabla d^s|\right)
	= d_t^s - \Delta d^s \geq 0
\]
and this inequality still holds for $\varepsilon$ negative so that it is valid in
\[
	W^\delta = \left\{ (x,t) \in E_{2\delta}^c \bigm|
	\left| d^s(x,t) \right| < 2\delta,\ 0 \leq t < T \right\}.
\]

Since this is an orientation free motion, this implies that $|d^s|$ is a viscosity supersolution of the level-set equation in $W^\delta$ with no obstacle .
 Since $|d^s|=\psi^+$ on $E_{2\delta}$, this implies that $|d^s|$ is a viscosity supersolution to the level-set equation with upper obstacle $\psi^+$ in 
\[
	U^\delta = \left\{ (x,t) \in \mathbb{R}^n \times (0,T) \bigm|
	|d^s| < 2\delta \right\}.
\]
By the invariance Lemma \ref{Linv}, we conclude that $w=|d^s|\wedge\delta$ is a viscosity supersolution of the level-set equation in $\mathbb{R}^n\times(0,T)$ with obstacle $\psi^+$.

For $\bar{d}=e^{-\lambda t}d^s$, we agree in the same way to obtain that
\begin{align}
\begin{aligned} \label{Edbar}
	\bar{d}_t - |\nabla\bar{d}| \operatorname{div} \left(\nabla\bar{d}/|\nabla\bar{d}|\right)
	&= e^{-\lambda t} (d_t^s - \Delta d^s) - \lambda e^{-\lambda t}d^s \\
	&= e^{-\lambda t}d^s \left( \sum_{i=1}^{n-1} \frac{\kappa_i^2}{1-d^s\kappa_i} - \lambda \right)
	\leq 0
\end{aligned}
\end{align}
in $W_+^\delta$.
 We take $\delta$ small so that $2\delta<\delta_0$ where $\delta_0$ is given in Lemma \ref{Lsigma}, which implies that
\[
	\bar{d}_t
	- |\nabla \bar{d}| \operatorname{div} \left(\nabla\bar{d}/|\nabla\bar{d}|\right) \leq 0
\]
in $E_{2\delta}$.
 We further take $\delta$ small so that $1-2\delta\kappa_i$ is bounded from below with some positive constant in $\Gamma_t^s$.
 We then take $\lambda$ sufficiently large so that the right-hand side of \eqref{Edbar} is negative in $W_+^\delta$.
 By this choice, we now conclude that $e^{-\lambda t}|d^s|$ is a viscosity subsolution of the level-set equation in $U^\delta$ with upper obstacle $\psi^+=\operatorname{dist}(x,\Sigma)$.
 Thus $z$ is a viscosity subsolution of the level-set equation in $\mathbb{R}^n\times(0,T)$ with upper obstacle $\psi^+$.
\end{proof}
\begin{proof}[Proof of Theorem {\rm\ref{Tcons}}]
Let $u$ be the viscosity solution of the level-set equation with upper obstacle $\psi^+=\operatorname{dist}(x,\Sigma)$ and initial data
\[
	u_0(x) = \operatorname{dist}(x,\Gamma_0) \wedge \delta.
\]
Let $z$ and $w$ be as in Lemma \ref{Lsubsuper}.
 By definition
\[
	z \leq u_0 \leq w \quad\text{at}\quad t = 0.
\]
Since $z$ and $w$ are respectively a viscosity sub- and supersolution, by Lemma \ref{Lcompa} in a big ball, we conclude that $z\leq u\leq w$ in $\mathbb{R}^n\times(0,T)$.
 This implies that $\Gamma_t^s=\Gamma_t$.
\end{proof}
%
%

\section{A few perspectives} \label{S5} 

\subsection{Stability} \label{SS51} 

By a standard argument by using half-relaxed limits, 
it is not difficult to prove that if the initial data $u_{0m}$ converges to $u_0\in BUC(\mathbb{R}^n)$ uniformly and the obstacle $\psi_m^+$ converges to $\psi^+$ uniformly, then the viscosity solution $u_m$ converges to $u$ in $\mathbb{R}^n\times[0,T)$ locally uniformly, where $u$ is the viscosity solution to the limit problem. 
 However, this only yields that for any $\delta>0$, there is $m_0$ such that if $m\geq m_0$, then $\Gamma_t^m\subset U_\delta(\Gamma_t)$, where $\Gamma_t^m=\left\{x\in\mathbb{R}^n\bigm|u_m(x,t)=0\right\}$.
 The inclusion $\Gamma_t\subset U_\delta(\Gamma_t^m)$ does not hold in general even without obstacles and without fattening phenomena. 
 This is because our $\Gamma_t$ may not separate two non-empty open sets. 
 Here is a simple example.
If $\Gamma_t$ separates two non-empty open sets and no fattening occurs, the convergence holds as in \cite[Section 4.6]{G}.
\begin{exam}
For $x\in\mathbb{R}^2$, we set the initial data $u_0(x)=\left|1-|x|\right|$ so that $\Gamma_0$ is a unit circle 
in $\mathbb{R}^2$. For $m\in\mathbb{N}$, let
\[
	u_{0m}(x) = \left( \frac1m - |x-P| \right)_+
	\vee u_0(x)
\]
and $P=(1,0)$, where $a_+$ denotes the plus part of $a$, i.e., $a_+=a\vee 0$. 
 By definition,
\[
	\Gamma_0^m = \Gamma_0 \backslash \left\{x\in\mathbb{R}^2 \biggm| |x-P| \ge \frac1m \right\}; 
\]
see Figure \ref{Fbcir}. 
\begin{figure}[htb]
\centering
\includegraphics[width=6cm]{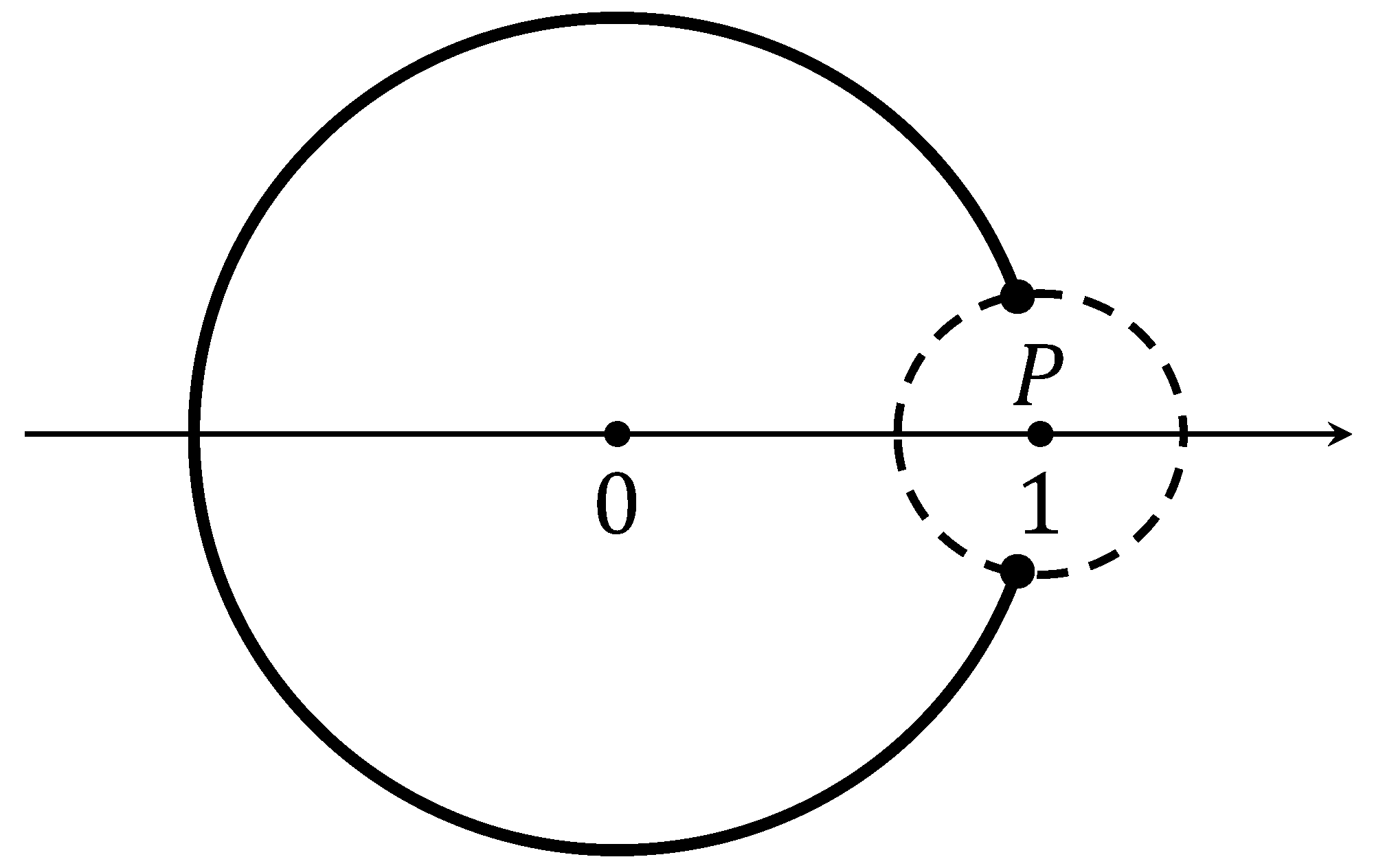}
\caption{figure of $\Gamma_0^m$ } \label{Fbcir}
\end{figure}
 The solution $\Gamma_t^m$ instantaneously disappears because $\Gamma_0^m$ is not closed.
 This can be proved as in \cite[Section 4.7]{G}.
 This is instant extinction.
 Evidently, the evolution
\[
	\Gamma_t = \left\{x\in\mathbb{R}^2 \bigm| |x| = \sqrt{1-2t} \right\}.
\]
cannot be approximated by $\Gamma_t^m$.
\end{exam}

If we approximate $\Sigma$ by $\Sigma_\delta=\overline{U_\delta(\Sigma)}$, the natural question is whether the level-set flow $\Gamma_t^\delta$ with obstacle $\Gamma_\delta$ approximates the level-set flow $\Gamma_t$ with obstacle $\Sigma$ as $\delta\downarrow0$.
 Because of the previous example, it is an interesting question to consider what kind of sequence of initial data $\Gamma_0^\delta\to\Gamma_0$ yields $\Gamma_t^\delta\to\Gamma_t$ in the sense of Hausdorff distance uniformly in $[0,T]$.
 Of course, one may consider other boundary conditions like right-angle condition.
 We may ask a similar convergence problem.

\subsection{Non orientation-free flow} \label{SS52} 

If we consider $V=H+\mathrm{const}$, the flow depends on the orientation.
 We cannot take nonnegative function to represent the level-set flow.
 Such an evolution with driving force is important, for example, in analysis of spirals caused by screw dislocations in a crystal surface \cite{OGT}.
 The approach by \cite{GNO} for spiral is promising since it considers the problem in a covering space although $\Sigma$ is a set of disjoint disks with right-angle boundary condition.
 
\subsection{Higher codimensional mean curvature flow} \label{SS53} 
In \cite{AS}, a level-set method for a motion of higher codimensional manifold in $\mathbb{R}^n$ by its mean curvature vector was established.
 The corresponding problem with a prescribed boundary is of the form
\begin{equation} \label{EMCDH}
    \left\{
    \begin{alignedat}{2} 
	\vec{V} &= \vec{H} \quad\text{on}\quad \Gamma_t, \quad t > 0, \\
	b\Gamma_t &= \Sigma, \quad t > 0, \\
		\left.\Gamma_t \right|_{t=0} &= \Gamma_0.
    \end{alignedat}
    \right.
\end{equation}
Here $\Sigma$ is a codimension $k+1$ submanifold and $\Gamma_t$ is a codimension $k$ submanifold in $\mathbb{R}^n$, where $k\geq2$;
 in the case $k=1$, the problem is nothing but \eqref{EMCD}.
 The motion is determined by the velocity vector $\vec{V}$ and the mean curvature vector $\vec{H}$.
 In \cite{AS}, they derived the level-set equation for a nonnegative function so that its zero level set evolves by $\vec{V}=\vec{H}$.
 We remark here that their theory easily extends to \eqref{EMCDH} by interpreting $\Sigma$ as an obstacle as in this paper.
 The consistency with smooth flow can be extended to this case by constructing viscosity sub- and supersolutions.
 One has to use $\eta=d^2/2$ instead of signed distance $d^s$.
 However, it is not clear what is the situation corresponding to the case studied by \cite{SZ}.
 We do not pursue this topic further in this paper. 



\begin{thebibliography}{AM}
%
\bibitem[Amb]{Amb}
L.~Ambrosio, 
Geometric evolution problems, distance function and viscosity solutions. 
\emph{Calculus of variations and partial differential equations (Pisa, 1996),} 5--93, 
\emph{Springer, Berlin,} 2000.
%
\bibitem[AM]{AM}
L.~Ambrosio and C.~Mantegazza, 
Curvature and distance function from a manifold. Dedicated to the memory of Fred Almgren. 
\emph{J.\ Geom.\ Anal.}\ 8 (1998), no.~5, 723--748.
%
\bibitem[AS]{AS}
L.~Ambrosio and H.~M.~Soner, 
Level set approach to mean curvature flow in arbitrary codimension. 
\emph{J.\ Differential Geom.}\ 43 (1996), no.~4, 693--737.
%
\bibitem[An]{An}
S.~B.~Angenent, 
Shrinking doughnuts. 
In: Nonlinear diffusion equations and their equilibrium states, eds., 
N.~G.~Lloyd, W.~M.~ Ni, L.~A.~ Peletier and J.~Serrin 3, 
\emph{Birkh\"auser, Basel-Boston-Berlin,} 1992. pp.\ 21--38.
%
\bibitem[BG]{BG}
G.~Barles and C.~Georgelin, 
A simple proof of convergence for an approximation scheme for computing motions by mean curvature. 
\emph{SIAM J.\ Numer.\ Anal.}\ 32 (1995), 484--500.
%
\bibitem[CGG]{CGG}
Y.-G.~Chen, Y.~Giga and S.~Goto, 
Uniqueness and existence of viscosity solutions of generalized mean curvature flow equations. 
\emph{J.\ Differential Geom.}\ 33 (1991), 749--786.
%
\bibitem[CIL]{CIL}
M.~G.~Crandall, H.~Ishii and P.-L.~Lions, 
User's guide to viscosity solutions of second order partial differential equations.
\emph{Bull.\ Amer.\ Math.\ Soc.}\ (N.S.) 27 (1992), 1--67.
%
\bibitem[ES]{ES}
L.~C.~Evans and J.~Spruck, 
Motion of level sets by mean curvature.\ I. 
\emph{J.\ Differential Geom.}\ 33 (1991), 635--681.
%
\bibitem[Fo]{Fo}
N.~Forcadel, 
Dislocation dynamics with a mean curvature term: short time existence and uniqueness.
\emph{Differential Integral Equations} 21 (2008), 285--304.
%
\bibitem[G]{G}
Y.~Giga, 
Surface Evolution Equations. A Level Set Approach. 
Monographs in Mathematics, 99. 
\emph{Birkh\"auser, Basel-Boston-Berlin,} 2006. x\hspace{-0.1em}i\hspace{-0.1em}i+264 pp.
%
\bibitem[GG1]{GG1} 
Y.~Giga and S.~Goto, 
Geometric evolution of phase-boundaries. 
\emph{On the evolution of phase boundaries (Minneapolis, MN, 1990--91),} 51--65, IMA Vol.\ Math.\ Appl., 43, 
\emph{Springer, New York,} 1992.
%
\bibitem[GG]{GG}
Y.~Giga and S.~Goto, 
Motion of hypersurfaces and geometric equations. 
\emph{J.\ Math.\ Soc.\ Japan} 44 (1992), 99--111.
%
\bibitem[GGIS]{GGIS}
Y.~Giga and S.~Goto, H.~Ishii and M.-H.~Sato, 
Comparison principle and convexity preserving properties for singular degenerate parabolic equations on unbounded domains.
\emph{Indiana Univ.\ Math.\ Journal} 40 (1991), 443--470.
%
\bibitem[GP]{GP}
Y.~Giga and N.~Po\v{z}\'ar, 
Viscosity solutions for the crystalline mean curvature flow with a nonuniform driving force term. 
\emph{SN Partial Differ.\ Equ.\ Appl.}\ 1 (2020), Article number: 39.
%
\bibitem[GTZ]{GTZ}
Y.~Giga, H.~V.~Tran and L.~Zhang, 
On obstacle problem for mean curvature flow with driving force. \emph{Geom.\ Flows} 4 (2019), 9--29.
%
\bibitem[GT]{GT}
D.~Gilbarg, N.~S.~Trudinger, 
Elliptic partial differential equations of second order. Reprint of the 1998 edition. Classics in Mathematics. 
\emph{Springer-Verlag, Berlin,} 2001. x\hspace{-0.1em}i\hspace{-0.1em}v+517 pp.
%
\bibitem[GNO]{GNO}
S.~Goto, M.~ Nakagawa and T.~Ohtsuka, 
Uniqueness and existence of generalized motion for spiral crystal growth. 
\emph{Indiana Univ.\ Math.\ J.}\ 57 (2008), no.\ 5, 2571--2599.
%
\bibitem[Gr87]{Gr87}
M.~A.~Grayson, 
The heat equation shrinks embedded plane curves to round points.
\emph{J.\ Differential Geom.}\ 26 (1987), no.\ 2, 285--314.
%
\bibitem[Gr89]{Gr89}
M.~A.~Grayson, 
A short note on the evolution of a surface by its mean curvature.
\emph{Duke Math.\ J.}\ 58 (1989), no.\ 3, 555--558.
%
\bibitem[KP]{KP}
S.~G.~Krantz and H.~R.~Parks, 
The implicit function theorem. History, theory, and applications. Reprint of the 2003 edition. 
\emph{Modern Birkh\"auser Classics. Birkh\"auser/Springer, New York,} 2013. x\hspace{-0.1em}i\hspace{-0.1em}v+163 pp.
%
\bibitem[M]{M}
G.~Mercier, 
Mean curvature flow with obstacles: a viscosity approach. 
\url{https://arxiv.org/abs/1409.7657}
%
\bibitem[OGT]{OGT}
T.~Ohtsuka, Y.-H.~R.~Tsai and Y.~Giga, 
A level set approach reflecting sheet structure with single auxiliary function for evolving spirals on crystal surfaces. 
\emph{J.\ Sci.\ Comput.}\ 62 (2015), no.\ 3, 831--874.
%
\bibitem[SZ]{SZ}
P.~Sternberg and W.~P.~Ziemer, 
Generalized motion by curvature with a Dirichlet condition. 
\emph{J.\ Differential Equations} 114 (1994), 580--600.
%
\bibitem[Ya]{Ya}
N.~Yamada, 
Viscosity solutions for a system of elliptic inequalities with bilateral obstacles. 
\emph{Funkcial.\ Ekvac.}\ 30 (1987), 417--425. 
%
\end{thebibliography}
\end{document}